\documentclass[11pt]{article}

\usepackage[margin=1in]{geometry}
\usepackage{setspace}
\usepackage{amsmath,amsfonts,amssymb,mathtools}
\usepackage{amsthm}

\usepackage{times}
\usepackage[varg]{txfonts}
\usepackage{bm}

\usepackage{algorithm}
\usepackage{algpseudocode}
\algblock{Indent}{EndIndent}
\algnotext{Indent}
\algnotext{EndIndent}
\algnotext{EndIf}
\algnotext{EndFor}

\usepackage{graphicx}
\usepackage{booktabs,multirow,array}
\usepackage{enumitem,needspace}
\usepackage{float}
\usepackage{xcolor}
\usepackage[font=small,labelfont=bf]{caption}

\usepackage{natbib}
\usepackage[hidelinks]{hyperref}
\bibpunct[, ]{(}{)}{,}{a}{}{,}
\hypersetup{pdftitle={Computing Equilibria in Integer Programming Games with Shared Constraints},
            pdfauthor={Bainian Hao, Hyunwoo Lee, Robert Hildebrand, Carla Michini}}
\newtheorem{theorem}{Theorem}
\newtheorem{proposition}{Proposition}
\newtheorem{lemma}{Lemma}
\newtheorem{corollary}{Corollary}
\newtheorem{definition}{Definition}
\newtheorem{assumption}{Assumption}

\theoremstyle{definition}
\newtheorem{example}{Example}
\newtheorem{remark}{Remark}

\DeclareMathOperator{\BR}{\mathcal{BR}}
\newcommand{\calX}{\mathcal{X}}

\newcommand{\R}{\mathbb{R}}
\newcommand{\Z}{\mathbb{Z}}

\newcommand{\free}{\mathrm{free}}

\newcommand{\xhat}{\hat{x}}

\begin{document}

\title{\Large\bfseries Computing Equilibria in Integer Programming Games\\with Shared Constraints}

\author{
Bainian Hao\thanks{School of Economics and Management, Chang'an University. Email: \texttt{bainianhao@chd.edu.cn}} \and
Hyunwoo Lee\thanks{Grado Department of Industrial and Systems Engineering, Virginia Tech. Email: \texttt{hyunwoolee@vt.edu}} \and
Robert Hildebrand\thanks{Grado Department of Industrial and Systems Engineering, Virginia Tech. Email: \texttt{rhil@vt.edu}}\and
Carla Michini\thanks{Department of Industrial and Systems Engineering, University of Wisconsin--Madison. Email: \texttt{carla.michini@wisc.edu}}
}

\date{\today}

\maketitle

\begin{abstract}
We develop a cutting-plane algorithm for computing pure Nash equilibria in finite integer games with shared constraints, where a deviation may be feasible against one opponent profile and infeasible against another. Conditional equilibrium inequalities capture this dependence through an explicit activation term. We give affine encodings of costs and activation conditions and prove that the resulting inequalities characterize the equilibrium set exactly. Embedded as lazy constraints in branch-and-cut, they yield the Generalized Zero-Regret algorithm for optimizing a linear objective over exact or approximate equilibria; a bisection procedure maintains certified bounds on the smallest achievable approximation factor. We also prove that, when the conditional polytopes have integral vertices, a concave cost is constant on the minimal face containing an equilibrium strategy, so strict concavity forces vertex strategies; for uniform integer-splittable bin packing, a cost-preserving transformation yields vertex equilibria even without strict concavity. We give formulations for bin packing, network formation, and knapsack games with shared capacities and evaluate the algorithm on 3{,}300 instances, computing a socially optimal equilibrium on all but 43, certifying nonexistence on 16, and certifying approximation factors within a few percent where no exact equilibrium exists.
\end{abstract}


\bigskip

\section{Introduction}
\label{sec:intro}

A \emph{generalized integer programming game} (GIPG) is a non-cooperative game in which each of $n$ players solves an integer program whose feasible set depends on the other players' decisions through coupling constraints.
GIPGs sit between two well-studied classes. Removing the coupling, so that each player's feasible set is independent of the opponents, recovers an \emph{integer programming game} (IPG) \citep{koppe2011rational,carvalho2018complexity}; a GIPG is thus an IPG with coupling constraints. It is also a \emph{generalized Nash equilibrium problem} (GNEP) all of whose variables are integer; GNEPs with mixed-integer variables have been studied by \citet{sagratella2017potential,sagratella2019coupling} and \citet{harks2025generalized}. Classical existence results for GNEPs assume convex compact strategy sets and costs convex in each player's own strategy, and neither holds here: the strategy sets are discrete, and the cost-sharing games below have concave costs.

Our solution concept is the \emph{pure Nash equilibrium} (PNE): a jointly feasible strategy profile from which no player can lower her cost by a unilateral, jointly feasible deviation. We also consider $\alpha$-approximate PNE, in which no player can lower her cost by more than a factor $\alpha\geq1$.

GIPGs arise naturally when discrete resources are shared.
In \emph{bin packing games}, players place items in capacitated bins and share each bin's cost in proportion to their usage \citep{bilo2006packing,hao2025integer}.
In \emph{generalized knapsack games}, players select items from a common pool subject to per-item capacity limits \citep{duguet2025branch1}.
In \emph{network formation games}, players route integer-splittable flows through a network with shared edge capacities and proportional cost sharing \citep{anshelevich2008price,chen2009network}.
In each case, one player's feasible strategies depend on how much of the shared capacity the other players have already consumed.

Computing a PNE of a GIPG is hard in a strong sense. Even for IPGs, deciding whether a PNE exists is $\Sigma_2^p$-complete \citep{carvalho2018complexity}, and GIPGs contain IPGs as a special case.
Shared constraints introduce an additional, more specific difficulty: the \emph{conditional feasibility} of deviations.
In an IPG, every privately feasible deviation is jointly feasible, so the equilibrium condition for a fixed candidate deviation is an unconditional cost comparison.
In a GIPG, the same deviation may be feasible against one opponent profile and infeasible against another, because the opponents may have exhausted the shared resources.
An inequality that enforces equilibrium must therefore be \emph{conditional}: it must impose the no-regret comparison exactly when the deviation is feasible against the current opponent strategies, and impose nothing otherwise.

This paper develops a cutting-plane algorithm for GIPGs built on \emph{conditional equilibrium inequalities} (CEIs): globally valid linear inequalities that encode the equilibrium condition together with an explicit activation term for conditional feasibility.
CEIs are separated by best-response computations and added as lazy constraints in a branch-and-cut solver, yielding the \emph{Generalized Zero-Regret} (GZR) algorithm for optimizing a linear objective over the set of exact or approximate equilibria.
We complement GZR with structural results on the \emph{vertex equilibrium structure} (VEST) of games with concave costs, which yield additional valid or optimality-preserving inequalities, and with two search procedures for approximate equilibria: random-restart best-response dynamics and a certified bisection on the approximation factor.
The algorithm is instantiated and tested on four GIPG classes: integer-splittable bin packing (ISBP), network formation games (NFG), generalized knapsack games (GKG), and generalized knapsack games with bilinear costs (GKGB).

\subsection{Literature Review}
\label{subsec:literature}
\citet{carvalho2023integer} survey computational methods for integer programming games. For pure equilibria, \citet{sagratella2016computing} develops a branching method for discrete strategy sets, and \citet{schwarze2023branch} give a branch-and-prune method for convex polyhedral strategy sets. Closest to our approach is the Zero-Regret algorithm of \citet{dragotto2023zero}, which embeds equilibrium inequalities in an integer program and thereby selects a pure equilibrium optimizing an external objective. For \emph{mixed} equilibria, \citet{cronert2024equilibrium} enumerate and select among the equilibria of finite games, and \citet{carvalho2026cutandplay} compute them by outer approximation of the convexified game. None of these methods admits shared constraints; GZR extends the approach of \citet{dragotto2023zero} to shared constraints by accounting explicitly for the conditional feasibility of deviations.

For mixed-integer GNEPs, \citet{harks2025generalized} characterize equilibria by convexification through the Nikaido--Isoda function, and \citet{sagratella2019coupling} gives a merit-function branch-and-bound method for linearly coupled games. \citet{duguet2025branch1} develop a branch-and-cut method on a high-point relaxation of the Nikaido--Isoda function, separating locally valid intersection cuts. In a companion paper, \citet{duguet2025branch2} treat approximate equilibria by bisecting the approximation factor, reusing the unexplored nodes and cuts of the previous value in a single tree.

For generalized potential games with mixed-integer variables, \citet{sagratella2017potential} proposes a Gauss--Southwell scheme that improves one player at a time, and \citet{fabiani2020vehicle} apply best-response algorithms to a mixed-integer potential game arising in vehicle coordination. Without a potential function, best-response dynamics may cycle; \citet{lee2024random} show that random restarts recover an almost-sure guarantee, finding a PNE with probability one whenever one is reachable from the restart law. We adapt their scheme to shared constraints.

Our structural results are related in spirit to \citet{rosen1965existence}, who studies games with concave payoffs maximized over fixed strategy sets. We instead consider the \emph{minimization} of concave costs over the integer points of opponent-dependent polytopes. The basic observation is elementary, namely that a concave function on a polytope attains its minimum at a vertex; integrality of the vertices is what allows this observation to be combined with the equilibrium conditions.

The applications combine discrete allocation with shared capacities. ISBP extends selfish bin packing to integer-splittable items \citep{hao2025integer}. Knapsack games are an established benchmark for equilibrium computation \citep{carvalho2016computation,carvalho2022computing,dragotto2023zero,duguet2025branch1}. Our network model uses integer-splittable flows with proportional cost sharing and edge capacities, and is related to the weighted capacitated model of \citet{Chassein2016}. Classical potential-function and existence results for network and congestion games \citep{rosenthal1973class,anshelevich2008price,chen2009network,feldman2012capacitated,fotakis2005selfish,delpia2017tu} rely on specific sharing rules and feasibility assumptions and do not automatically carry over to this setting; indeed, our network model contains ISBP instances without a PNE.

\subsection{Contributions}
\label{subsec:contributions}
The central difficulty in a GIPG is that whether a deviation is available to a player depends on what the other players are doing. We show that this dependence can be handled within a single mixed-integer linear program. The key is a family of \emph{conditional equilibrium inequalities}: each compares the cost of a player's current strategy with the cost of a fixed alternative, and is automatically switched off whenever that alternative is infeasible against the current opponents. Both the costs and the switch are expressed as affine functions of the same lifted variables that linearize the players' costs, so the inequalities are valid throughout the branch-and-bound tree and can be generated on the fly by solving best-response problems. We prove that the inequalities characterize the set of equilibria exactly, and we obtain the Generalized Zero-Regret algorithm, which optimizes any linear objective over exact or approximate equilibria or certifies that none exists. A bisection on the approximation factor, seeded by best-response dynamics, produces certified bounds on how close to an exact equilibrium one can get when none exists.

Second, we identify structures that equilibria must have when costs are concave, as in cost-sharing games: a player's equilibrium strategy is either a vertex of her feasible polytope, or the player is indifferent among all strategies on the face containing it. When costs are strictly concave this forces every equilibrium strategy to a vertex, and for integer-splittable bin packing with uniform bins we show that vertex equilibria exist even though strict concavity fails. These results translate into additional inequalities for the master problem, and we make precise which of them are valid for every equilibrium and which only preserve an optimal one, a distinction that matters whenever nonexistence or optimality is certified.

Third, we instantiate the algorithm for four classes of games, namely bin packing, network formation, and knapsack games with linear and with bilinear costs, and test it on 3{,}300 instances. GZR computes socially optimal equilibria and certifies nonexistence within minutes on most instances (in particular, it certifies that 15 regular bin packing instances admit no pure equilibrium, where one was previously known), the bisection procedure closes the approximation gap to within a few percent where no exact equilibrium exists, and the structural cuts give a measurable speedup where the stronger restriction is available. We also observe that the linear generalized knapsack benchmark of \citet{duguet2025branch1} is an exact potential game, whose social optimum is already an equilibrium; it therefore measures the cost of separating equilibrium constraints rather than the value of equilibrium selection.

Section~\ref{sec:cei} develops CEIs and GZR, Section~\ref{sec:vest} the structural results, and Section~\ref{sec:tas} the complementary search procedures. Section~\ref{sec:games} formulates the applications, Section~\ref{sec:experiments} reports the computational study, and Section~\ref{sec:conclusion} concludes. Omitted proofs and formulations are in the e-companion.

\section{Conditional Equilibrium Inequalities}
\label{sec:cei}

We consider finite, jointly constrained integer games.
Players are indexed by $N=\{1,\ldots,n\}$; player $i$ controls a vector $x^i\in\Z^{k_i}$ of integer variables, and $K=\sum_ik_i$. A strategy profile is $x=(x^1,\ldots,x^n)$, and $x^{-i}$ denotes the profile of the opponents of player $i$. Let $\calX\subseteq\Z^K$ be the set of jointly feasible profiles, defined by private constraints on each $x^i$ and by shared constraints coupling the players; we assume that $\calX$ is finite and nonempty. Let $\calX_i=\operatorname{proj}_{x^i}(\calX)$ and $\calX_{-i}=\operatorname{proj}_{x^{-i}}(\calX)$, and for $x^{-i}\in\calX_{-i}$ let
\[
\calX_i(x^{-i})=\{x^i\in\Z^{k_i}:(x^i,x^{-i})\in\calX\}
\]
be the set of strategies available to player $i$ against $x^{-i}$. Every strategy in $\calX_i(x^{-i})$ yields a jointly feasible profile. When $\calX_i(x^{-i})$ does not depend on $x^{-i}$, the game is an IPG.

Player $i$ has a nonnegative cost $\pi^i:\calX\to\R_{\geq0}$ and solves
\begin{equation}
\label{eq:player-problem}
\min\{\pi^i(x^i,x^{-i}):x^i\in\calX_i(x^{-i})\}.
\end{equation}
Nonnegativity is part of the model: the multiplicative approximation notion introduced next depends on the cost scale, and it is meaningful only for nonnegative costs.

\begin{definition}[$\alpha$-approximate PNE]
\label{def:alpha-pne}
Fix $\alpha \geq 1$.
A strategy profile $\hat{x} \in \calX$ is an \emph{$\alpha$-approximate pure Nash equilibrium} ($\alpha$-PNE) if for every player $i \in N$,
\[
\pi^i(\hat{x}^i, \hat{x}^{-i}) \leq \alpha \cdot \pi^i(x^i, \hat{x}^{-i}) \qquad \forall x^i \in \calX_i(\hat{x}^{-i}).
\]
When $\alpha = 1$, this reduces to an exact PNE.
\end{definition}

\begin{definition}[$\alpha$-approximate best response]
\label{def:alpha-br}
Fix $\alpha \geq 1$ and $\hat{x}^{-i}$.
A strategy $\hat{x}^i \in \calX_i(\hat{x}^{-i})$ is an \emph{$\alpha$-approximate best response} ($\alpha$-BR) against $\hat{x}^{-i}$ if
\[
\pi^i(\hat{x}^i, \hat{x}^{-i}) \leq \alpha \cdot \pi^i(x^i, \hat{x}^{-i}) \qquad \forall x^i \in \calX_i(\hat{x}^{-i}).
\]
Let $\BR_i^\alpha(\hat{x}^{-i})$ denote the set of all $\alpha$-BRs of player $i$ against $\hat{x}^{-i}$.
\end{definition}

For a fixed candidate deviation $\hat x^i\in\calX_i$, equilibrium requires the implication
\begin{equation}
\label{eq:cei-implication}
\hat x^i\in\calX_i(x^{-i})\quad\Longrightarrow\quad
\pi^i(x)\leq\alpha\pi^i(\hat x^i,x^{-i}).
\end{equation}
The following activation convention permits a single inequality to express it.
\begin{definition}[Activation function]
\label{def:activation}
An activation function for the candidate $\hat x^i$ satisfies
\begin{equation}
\label{eq:delta}
\delta_i(\hat x^i,x^{-i})=0\ \text{if }\hat x^i\in\calX_i(x^{-i}),
\qquad
\delta_i(\hat x^i,x^{-i})\geq1\ \text{otherwise}.
\end{equation}
\end{definition}
Any function with these two properties can serve as activation function. Two natural choices are the indicator of infeasibility, equal to one whenever $\hat x^i\notin\calX_i(x^{-i})$, and the number of shared constraints violated by the profile $(\hat x^i,x^{-i})$; the encodings of Section~\ref{sec:games} produce the latter.
We extend $\pi^i$ to take an arbitrary nonnegative value whenever the deviation is not jointly feasible, that is, whenever $(\hat x^i,x^{-i})\notin\calX$. For $M_i\geq\max_{x\in\calX}\pi^i(x)$, the \emph{conditional equilibrium inequality} (CEI) for $\hat x^i$ is
\begin{equation}
\label{eq:cei}
\pi^i(x)\leq\alpha\pi^i(\hat x^i,x^{-i})+M_i\delta_i(\hat x^i,x^{-i}).
\end{equation}
It enforces the cost comparison when the deviation is feasible and is automatically satisfied otherwise. The modeling task is to represent all three functions in~\eqref{eq:cei} affinely in a common set of variables (Section~\ref{subsec:lifted}); the following example shows why strategy variables alone are not enough.

\begin{example}[A conditional deviation]
Consider two players with integer decisions $0\leq x^1\leq5$ and $1\leq x^2\leq4$ coupled by the shared constraint $x^1+x^2\leq6$, and let $\pi^1(x)=(6-x^1)x^2$. The candidate deviation $\hat x^1=4$ is feasible exactly when $x^2\leq2$, in which case equilibrium requires $\pi^1(x)\leq2\alpha x^2$; when $x^2\in\{3,4\}$ the deviation is infeasible and must impose no restriction. A binary indicator for the event $x^2>2$ and an exact linearization of the product $x^1x^2$ turn these two cases into a single linear CEI.
\end{example}

\subsection{Lifted Formulation}
\label{subsec:lifted}
Let $w\in\R^q$ collect auxiliary variables, with prescribed domain $W$ specifying which coordinates are integer or binary. We refer to $(x,w)$ as the \emph{master variables}, since they are the variables of the master problem solved by GZR. Let $\mathcal H\subseteq\R^{K+q}$ be a rational polyhedron and define
\[
\mathcal H_{\mathrm{MI}}:=\mathcal H\cap(\Z^K\times W),
\qquad \operatorname{proj}_x(\mathcal H_{\mathrm{MI}})=\calX.
\]
For each player we require an affine function $\tilde\pi^i$, the \emph{lifted cost}, such that
\[
\tilde\pi^i(x,w)=\pi^i(x)\qquad\text{for every }(x,w)\in\mathcal H_{\mathrm{MI}}.
\]
Feasibility of a deviation is then characterized in the lifted space by
\begin{equation}
\label{eq:delta-lifted}
\delta_i(\hat x^i,x^{-i})=0
\quad\Longleftrightarrow\quad
\exists\hat w\in W:(\hat x^i,x^{-i},\hat w)\in\mathcal H_{\mathrm{MI}}.
\end{equation}
The lifted equilibrium set is
\(
\mathcal E^\alpha(\mathcal G)=\{(x,w)\in\mathcal H_{\mathrm{MI}}:x\text{ is an }\alpha\text{-PNE}\}.
\)

\subsection{CEI Exactness}
\label{subsec:validity}
We now collect the main properties of the CEI system. To embed the CEI~\eqref{eq:cei} in the master problem, both terms on its right-hand side must be affine in the master variables for each fixed deviation $\hat x^i$. This is the content of the following assumption.

\begin{assumption}[Affine CEI representation]
\label{ass:cei-computability}
For every player $i \in N$ and every candidate deviation $\hat{x}^i \in \calX_i$ there exist affine functions $\tilde{\pi}^i_{\hat{x}^i},\tilde\delta_{i,\hat x^i} : \R^{K+q} \to \R$, the \emph{lifted deviation cost} and the \emph{lifted activation function}, such that for every $(x,w)\in\mathcal H_{\mathrm{MI}}$:
(i)~$\tilde{\pi}^i_{\hat{x}^i}(x, w) = \pi^i(\hat{x}^i, x^{-i})$ whenever $\hat{x}^i \in \calX_i(x^{-i})$;
(ii)~$\tilde{\pi}^i_{\hat{x}^i}(x, w) \geq 0$; and
(iii)~$\tilde\delta_{i,\hat x^i}(x,w)=\delta_i(\hat x^i,x^{-i})$.
\end{assumption}

In the simplest case, $\tilde{\pi}^i_{\hat{x}^i}$ is obtained by substituting $\hat{x}^i$ into $\tilde{\pi}^i$, as in GKG. For the one-hot ISBP and NFG encodings it is a separate affine function whose coefficients are computed from $\hat{x}^i$ at separation time using the joint state information carried by $w$; the explicit formulas are~\eqref{eq:supp-isbp-cei} and~\eqref{eq:supp-nfg-cei} in \S\ref{app:explicit-formulations}. The GKGB formula appears in~\eqref{eq:gkgb-cei}.
All four game classes in this paper satisfy the assumption by construction; see Section~\ref{sec:games}.

Note that the lifted deviation cost is an affine function of the master variables $(x,w)$ with $\hat x^i$ held fixed: it evaluates the cost of the deviation $\hat x^i$ against the opponent profile $x^{-i}$ encoded by $(x,w)$.

For a fixed $\hat x^i\in\calX_i$, the lifted CEI is
\begin{equation}
\label{eq:cei-lifted}
\tilde\pi^i(x,w)\leq\alpha\tilde\pi^i_{\hat x^i}(x,w)
+M_i\tilde\delta_{i,\hat x^i}(x,w).
\end{equation}
Let $\mathcal E^\alpha_{\mathrm{cei}}$ consist of the points of $\mathcal H_{\mathrm{MI}}$ satisfying~\eqref{eq:cei-lifted} for every $i$ and $\hat x^i\in\calX_i$, and let $P^\alpha_{\mathrm{cei}}=\operatorname{conv}(\mathcal E^\alpha_{\mathrm{cei}})$.

\begin{proposition}[CEI exactness and closure]\label{prop:cei}
Under Assumption~\ref{ass:cei-computability}, $\mathcal E^\alpha_{\mathrm{cei}}=\mathcal E^\alpha(\mathcal G)$ for every $\alpha\geq1$. If moreover each $x\in\calX$ has a unique auxiliary representation in $\mathcal H_{\mathrm{MI}}$, then $P^\alpha_{\mathrm{cei}}=\operatorname{conv}(\mathcal E^\alpha(\mathcal G))$ is a rational polytope whose vertices lie in $\mathcal E^\alpha(\mathcal G)$ and whose points in $\Z^K\times W$ are lifted $\alpha$-PNE; in particular, if $\mathcal E^\alpha(\mathcal G)\neq\emptyset$, every linear objective attains its minimum over $P^\alpha_{\mathrm{cei}}$ at a lifted $\alpha$-PNE.
\end{proposition}
The proof is in \S\ref{app:proof-cei}. The uniqueness hypothesis only ensures that $\mathcal E^\alpha(\mathcal G)$ is finite, so that its convex hull is a polytope; all encodings of Section~\ref{sec:games} satisfy it, since $w$ is determined by $x$ in each of them. Proposition~\ref{prop:cei} extends the equilibrium-inequality formulation of \citet{dragotto2023zero}, in which every deviation is unconditionally feasible, to games with shared constraints.

\subsection{Generalized Zero-Regret Algorithm}
\label{subsec:gzr}

Let $F(x,w)$ be a linear objective. GZR seeks an $F$-optimal $\alpha$-PNE or proves that none exists. At $\alpha=1$, minimizing $F=\sum_i\tilde\pi^i$ gives the minimum social cost over all PNE, which is the numerator of the cost-based \emph{price of stability} (POS); the denominator is the unrestricted minimum social cost, which must be positive for the ratio to be defined.

GZR (Algorithm~\ref{alg:gzr}) solves a master problem over $\mathcal H_{\mathrm{MI}}$ by branch-and-cut and checks each candidate $(\bar x,\bar w)$ by solving one best-response problem per player. Each player whose best response improves on $\bar x^i$ by more than the factor $\alpha$ generates a CEI, which is added as a lazy constraint and cuts off the candidate. If no player has such a deviation and the best-response problems are solved to optimality, the candidate is an $\alpha$-PNE and becomes the incumbent. Algorithm~\ref{alg:gzr} terminates only when the branch-and-cut of Step~1 does; Steps~\ref{gzr:cb-start}--\ref{gzr:cb-end} either discard the current candidate by adding violated constraints or record it as the incumbent. In our applications the best-response problem in Step~\ref{gzr:br} is a MILP: a knapsack problem for GKG and GKGB, and for ISBP and NFG, a reduced one-hot formulation in which the opponents' loads are fixed.

\begin{algorithm}[t]
\small
\caption{Generalized Zero-Regret (GZR)}\label{alg:gzr}
\begin{algorithmic}[1]
\Require $\mathcal H_{\mathrm{MI}}$, linear objective $F$, $\alpha\geq1$, a set $\Phi_0$ of inequalities satisfied by every point of $\mathcal E^\alpha(\mathcal G)$
\State Solve $\min\{F(x,w):(x,w)\in\mathcal H_{\mathrm{MI}}\text{ satisfying }\Phi_0\}$ by branch-and-cut; whenever the solver finds a mixed-integer feasible solution $(\bar x,\bar w)$, called a \emph{candidate}, execute Steps~\ref{gzr:cb-start}--\ref{gzr:cb-end} before accepting it as incumbent:
\Indent
\State $\mathit{cut}\gets\textsc{false}$ \label{gzr:cb-start}
\For{$i\in N$}
  \State Solve the best-response problem~\eqref{eq:player-problem} of player $i$ against $\bar x^{-i}$; let $\hat x^i$ be an optimal solution. \label{gzr:br}
  \If{$\pi^i(\bar x)>\alpha\,\pi^i(\hat x^i,\bar x^{-i})$}
    \State Add the CEI~\eqref{eq:cei-lifted} for $(i,\hat x^i)$ as a lazy constraint (the candidate violates it and is discarded); $\mathit{cut}\gets\textsc{true}$.
  \EndIf
\EndFor
\State \textbf{if} $\mathit{cut}=\textsc{false}$ \textbf{then} accept $(\bar x,\bar w)$ as the new incumbent. \label{gzr:cb-end}
\EndIndent
\State \Return the final incumbent, if any, together with the status \textsc{Optimal} if branch-and-cut terminated with an incumbent, \textsc{Infeasible} if it proved the master problem infeasible, and \textsc{Unresolved} otherwise.
\end{algorithmic}
\end{algorithm}

\begin{proposition}[GZR correctness]\label{prop:gzr-correct}
Suppose that every best-response problem in Step~\ref{gzr:br} of Algorithm~\ref{alg:gzr} is solved to optimality and that every inequality in $\Phi_0$ is satisfied by all points of $\mathcal E^\alpha(\mathcal G)$. Then every candidate accepted in Step~\ref{gzr:cb-end} is an $\alpha$-PNE; if the algorithm terminates with status \textsc{Optimal}, the final incumbent minimizes $F$ over $\mathcal E^\alpha(\mathcal G)$; and if it terminates with status \textsc{Infeasible}, then $\mathcal E^\alpha(\mathcal G)=\emptyset$.
\end{proposition}
The proof is in \S\ref{app:proof-gzr}. If the search is stopped before completion and no profile has been accepted, neither the existence nor the nonexistence of an $\alpha$-PNE can be concluded. Since $\calX$ is finite and each separated CEI permanently rejects a nonequilibrium profile, GZR terminates after finitely many separations. These guarantees assume optimal best responses; Section~\ref{sec:experiments} discusses how they relate to solver tolerances in the implementation.

Additional restrictions on the master require care. Inequalities valid for all equilibria, such as the weak VEST inequalities of Section~\ref{sec:vest}, can be added directly. A restriction that preserves only some equilibrium representatives must also preserve an optimum for the chosen objective $F$, and imposing several such restrictions jointly requires a simultaneous preservation argument; \S\ref{app:symmetry} provides one for ISBP.

\section{Vertex Equilibrium Structure}
\label{sec:vest}

This section develops structural properties of equilibria in games whose costs are concave in each player's own strategy, as is the case for the proportional cost-sharing games of Section~\ref{sec:games}. Fix a jointly feasible opponent profile $x^{-i}$, let $P^i(x^{-i})$ be a polytope whose integer points are exactly $\calX_i(x^{-i})$, and extend $\pi^i(\cdot,x^{-i})$ to this polytope. The main result, face constancy under concavity and integral vertices, holds for arbitrary conditional polytopes. For the explicit inequalities derived later we specialize to the equality-and-bound form
\[
P^i(x^{-i})=\{x^i\in\R^{k_i}:A^ix^i=b^i(x^{-i}),\ 0\leq x^i_j\leq u_{ij}(x^{-i})\}.
\]
and write $P^i$ when the opponents' profile is clear from context. We show that strict concavity forces equilibrium strategies to vertices, while separability restricts the coordinates that can vary along a nontrivial equilibrium face. Omitted proofs for this section and Section~\ref{sec:games} are in \S\ref{app:moved-proofs}.

\begin{assumption}[Concave Costs]
\label{ass:concave}
For each player $i \in N$ and each opponents' profile $x^{-i}$,
the cost function $\pi^i(\cdot, x^{-i})$ is concave on $P^i(x^{-i})$.
\end{assumption}

\begin{assumption}[Integer Vertices]
\label{ass:integer-vertices}
For each player $i \in N$ and each opponents' profile $x^{-i}$,
every vertex of $P^i(x^{-i})$ is an integer point.
\end{assumption}

Assumption~\ref{ass:integer-vertices} holds whenever the constraint matrix
$A^i$ is totally unimodular and $b^i(x^{-i})$, $u_{ij}(x^{-i})$ are
integer for every $x^{-i}$.
This is the case for the two cost-sharing games considered in this paper.

For a point $x^i$ of a polytope, a \emph{direction within its minimal face} is a nonzero vector $d$ such that $x^i\pm\varepsilon d$ lie in that face for some $\varepsilon>0$. For the equality-and-bound representation, define the set of \emph{free coordinates}
\[
\free(x^i)=\{j:0<x^i_j<u_{ij}(x^{-i})\}.
\]
The directions within the minimal face are precisely the nonzero $d$ satisfying $A^id=0$ and $d_j=0$ outside $\free(x^i)$. A point is a vertex exactly when no such direction exists.

\begin{theorem}[Weak VEST]
\label{thm:weak-vest}
Under Assumptions~\ref{ass:concave} and~\ref{ass:integer-vertices}, if $x$ is a PNE, then for every player $i$ the cost $\pi^i(\cdot, x^{-i})$ is constant on the minimal face of $P^i(x^{-i})$ containing $x^i$.
\end{theorem}
\begin{proof}
Fix $i$, write $f=\pi^i(\cdot,x^{-i})$, and let $F$ be the minimal face of $P^i(x^{-i})$ containing $x^i$. If $x^i$ is a vertex, then $F=\{x^i\}$ and the claim holds trivially. Suppose $F$ has dimension at least one and let $y\in F$, $y\neq x^i$.

We first show that $x^i$ minimizes $f$ over $P^i(x^{-i})$. By Assumption~\ref{ass:integer-vertices} every vertex $v$ of $P^i(x^{-i})$ is an integer point, hence a feasible deviation, so $f(v)\geq f(x^i)$ because $x$ is a PNE. Every point of $P^i(x^{-i})$ is a convex combination $\sum_k\lambda_kv_k$ of vertices, and concavity gives $f(\sum_k\lambda_kv_k)\geq\sum_k\lambda_kf(v_k)\geq f(x^i)$.

Since $x^i$ lies in the relative interior of $F$, the point $z=x^i+\varepsilon(x^i-y)$ belongs to $F$ for some $\varepsilon>0$, and $x^i=\lambda y+(1-\lambda)z$ with $\lambda=\varepsilon/(1+\varepsilon)\in(0,1)$. By concavity and minimality of $x^i$, $f(x^i)\geq\lambda f(y)+(1-\lambda)f(z)\geq f(x^i)$, so equality holds throughout and $f(y)=f(x^i)$.
\end{proof}

Theorem~\ref{thm:weak-vest} says that a nonvertex equilibrium strategy can only occur where the player is indifferent along an entire face. Strengthening Assumption~\ref{ass:concave} to strict concavity rules this out and pins every PNE strategy to a vertex.

\begin{assumption}[Strictly Concave Costs]
\label{ass:strictly-concave}
For each player $i \in N$ and each opponents' profile $x^{-i}$,
the cost function $\pi^i(\cdot, x^{-i})$ is strictly concave on $P^i(x^{-i})$.
\end{assumption}

\begin{corollary}[Strong VEST]
\label{thm:strong-vest}
Under Assumptions~\ref{ass:strictly-concave} and~\ref{ass:integer-vertices},
if $x$ is a PNE then $x^i$ is a vertex of $P^i(x^{-i})$ for every $i \in N$.
\end{corollary}

\begin{proof}
Suppose $x^i$ is not a vertex of $P^i(x^{-i})$. By Theorem~\ref{thm:weak-vest}, $\pi^i(\cdot, x^{-i})$ is constant on the minimal face $F$ of $P^i(x^{-i})$ containing $x^i$, and $\dim(F) \geq 1$. A strictly concave function cannot be constant on a segment, contradicting Assumption~\ref{ass:strictly-concave}. Hence $\dim(F) = 0$, i.e., $x^i$ is a vertex.
\end{proof}

When the cost is separable, $\pi^i(x^i,x^{-i})=\sum_{j} f^i_j(x^i_j;x^{-i})$,
Theorem~\ref{thm:weak-vest} can be sharpened even without strict concavity. For $x^i\in P^i(x^{-i})$ let
$S=\free(x^i)$, and partition $S$ according to the local behavior of each $f^i_j$ at
$x^i_j$:
\[
\begin{gathered}
S^i_0(x^i)=\bigl\{\,j\in S:\ f^i_j(\,\cdot\,;x^{-i})\text{ is affine on a neighborhood of }
x^i_j\text{ in }[0,u_{ij}(x^{-i})]\,\bigr\},
\\
S^i_+(x^i)=S\setminus S^i_0(x^i).
\end{gathered}
\]

These neighborhoods are well defined since free coordinates are interior to their bound intervals.

\begin{corollary}[Weak VEST for separable functions]\label{thm:weak-vest-separable}
Suppose $\pi^i(x^i,x^{-i})=\sum_{j} f^i_j(x^i_j;x^{-i})$ for all $i$, where each
$f^i_j(\,\cdot\,;x^{-i})$ is concave. Under Assumption~\ref{ass:integer-vertices}, if $x$ is a
PNE and $x^i$ is not a vertex of $P^i(x^{-i})$, then every feasible direction $d$ at $x^i$
within the minimal face $F$ of $P^i(x^{-i})$ containing $x^i$ satisfies $d_j=0$ for all
$j\in S^i_+(x^i)$: the only coordinates that can vary within $F$ are those in $S^i_0(x^i)$.
\end{corollary}

\section{Complementary Solution Algorithms}
\label{sec:tas}

GZR is complemented by two procedures: a best-response heuristic that supplies warm starts and initial approximation factors, and a bisection on $\alpha$ for instances without an exact PNE.

We adapt the random-restart best-response dynamics (RRR-BRD) of \citet{lee2024random} to shared constraints. Each attempt starts from a jointly feasible profile. A round visits the players in random order and lets each apply a strictly improving best response computed against the current profile; because best responses are drawn from the conditional strategy sets, every update preserves joint feasibility. A round with no updates certifies a PNE under optimal best responses, since every player was then checked against the same unchanged profile. After a prescribed number of rounds without such a certificate, the attempt restarts from a fresh random profile.

For an arbitrary feasible profile $x$, let $b_i(x^{-i})=\min_{x^i\in\calX_i(x^{-i})}\pi^i(x^i,x^{-i})$ be the best-response value of player $i$, and define
\[
\alpha(x)=\max\{1,\max_i\pi^i(x)/b_i(x^{-i})\},
\]
with the conventions $0/0=1$ and $c/0=+\infty$ for $c>0$. If the $b_i$ are evaluated exactly, $x$ is an $\alpha(x)$-PNE whenever $\alpha(x)$ is finite; this is how RRR-BRD supplies the initial approximation factor for the bisection described next. All best responses in this evaluation must be computed against the same final profile $x$; values obtained earlier in a round, while the profile was still changing, cannot in general be reused for certification.

When an exact PNE need not exist, we search for the smallest approximation factor $\alpha^*=\inf\{\alpha\geq1:\text{an }\alpha\text{-PNE exists}\}$; since $\calX$ is finite, the infimum is attained whenever it is finite. Because costs are nonnegative, $\BR_i^{\alpha_1}(x^{-i})\subseteq\BR_i^{\alpha_2}(x^{-i})$ for $1\leq\alpha_1\leq\alpha_2$, so the set of $\alpha$-PNE is monotone in $\alpha$ and bisection on $\alpha$ is meaningful. Monotonicity also allows cuts to be carried across bisection levels, with one caveat: a CEI generated at $\alpha_2$ remains valid at every $\alpha_1\leq\alpha_2$ but not necessarily at larger $\alpha$, so what is retained is the generating deviation $\hat x^i$, from which~\eqref{eq:cei-lifted} is re-instantiated at the current $\alpha$. \citet{duguet2025branch2} exploit the same monotonicity.

The procedure, which we call \emph{$\alpha$-bisection search} (ABS), maintains certified bounds $\alpha_L\leq\alpha^*\leq\alpha_U$ together with an $\alpha_U$-PNE $x^*$. Initially $\alpha_L=1$, and RRR-BRD supplies $\alpha_U$ and $x^*$ whenever its final-profile evaluation yields a finite $\alpha(x)$. At each iteration, GZR is run at $\alpha=(\alpha_L+\alpha_U)/2$ with the retained deviations instantiated at $\alpha$. If GZR accepts an $\alpha$-PNE $x$, then $\alpha_U\gets\alpha$ and $x^*\gets x$; the incumbent need not be optimal for this update. If GZR terminates \textsc{Infeasible}, then $\alpha_L\gets\alpha$. If GZR ends \textsc{Unresolved}, neither certified bound changes; in this case our implementation raises the endpoint of a separate \emph{working interval} so that a new trial value can be selected, which allows the search to continue or stop without closing the certified interval. The search stops when $\alpha_U-\alpha_L\leq\varepsilon$ or the resource budget is exhausted, and returns $(\alpha_L,\alpha_U,x^*)$. If every test is resolved, the certified width halves at each iteration, so at most $\lceil\log_2((\alpha_U^{(0)}-\alpha_L^{(0)})/\varepsilon)\rceil$ GZR calls are needed. Inequalities proved only for exact equilibria, such as the strong VEST cut~\eqref{eq:vest-cut}, are omitted at $\alpha>1$.

\section{Applications}
\label{sec:games}

\subsection{Integer Splittable Bin Packing (ISBP)}
\label{subsec:isbp}

Integer-splittable bin packing (ISBP) games, introduced by \citet{hao2025integer},
model decentralized allocation problems in which self-interested
agents share capacitated resources and benefit from cost sharing, as in
logistics, inventory allocation, and distributed resource management.
They are variants of the classical selfish bin packing games of
\citet{bilo2006packing}. In an ISBP game, each player $i \in [n]$ controls an item of
integer size $w_i$ and decides how to split this item into integer units
across $m$ bins. Each bin $j \in [m]$ has integer capacity $u_j \in \Z^+$ and constant cost $c_j > 0$, shared proportionally among the players using the bin.
The strategy of player $i$ is an integer vector $x^i = (x^i_j)_{j \in [m]}$,
where $x^i_j$ denotes the number of units of item $i$ assigned to bin
$j$.
Let $x_j = \sum_{i \in [n]} x^i_j$ denote the total load on bin $j$,
and $x^{-i}_j = \sum_{k \neq i} x^k_j$ the load due to the opponents
of player $i$.
The allocation must assign all units of item $i$ and must not
violate the bin capacities.
Throughout, we use the convention that the cost contribution of an empty bin is zero, i.e., $c_j x^i_j/(x^{-i}_j + x^i_j) := 0$ whenever $x^i_j = 0$.
Given the opponents' loads $x^{-i}$, player
$i$ solves:
\begin{equation}
\label{eq:isbp-player}
\min_{x^i}
\left\{
\sum_{j \in [m]}
\frac{c_j}{x_j^{-i}+x^i_{j}}\,x^i_{j}
:
\sum_{j \in [m]} x^i_{j}=w_i,\;
x_j^{-i}+x^i_{j}\le u_j \ \forall j\in [m],\;
x^i_{j}\in\Z_{\ge 0}\ \forall j\in [m]
\right\}.
\end{equation}
An ISBP game instance is called \emph{uniform} if all bins have the same
capacity, and \emph{non-uniform} if bins may have different capacities.
Moreover, a uniform instance is called \emph{regular} if all bins have
the same cost and all items have the same size. Lemma~1 of \citet{hao2025integer} states that every PNE of a uniform instance occupies exactly $\lceil\sum_iw_i/u\rceil$ bins. In regular instances the bin costs are identical, so every PNE is socially optimal and the POS equals one whenever a PNE exists.
\citet{hao2025integer} also showed that computing a socially optimal state is NP-hard for
non-uniform instances. Since a single-player instance is a best-response problem, the non-uniform class contains NP-hard best-response problems as well.

\subsubsection{One-hot Encoding and Lifted Formulation}
\label{ss:isbp-form}

The cost in~\eqref{eq:isbp-player} is a ratio of decision variables, so we linearize it through a one-hot encoding of joint states. A binary variable $y^{j,t}_{i,s}$ indicates that player $i$ assigns $s$ units to bin $j$ while the total load of bin $j$ is $t$, and a binary $z^j_t$ indicates that the total load of bin $j$ is $t$. Consistency constraints select one state per player--bin pair and per bin, tie the selected total load to the sum of the individual contributions, and impose item sizes and capacities; the cost then becomes the affine function $\sum_{j,t,s}c_j(s/t)\,y^{j,t}_{i,s}$. The encoding is exact on binary $(y,z)$ and uses $O(nmuw)$ variables for $u=\max_ju_j$ and $w=\max_iw_i$; the explicit formulation is given in \S\ref{app: ISBP formulation}.

\subsubsection{CEI Encoding and Activation}
\label{ss:isbp-cei}

For a candidate deviation $\hat x^i$, the lifted CEI~\eqref{eq:supp-isbp-cei} has the incumbent one-hot cost on the left-hand side. On the right-hand side, the deviation cost on bin $j$ is $c_j\hat x^i_j/(t-s+\hat x^i_j)$, summed over the states $(s,t)$ for which the deviation fits, that is, $t-s+\hat x^i_j\le u_j$; the big-$M$ term is activated by the states of any bin used by the deviation for which it does not fit. For example, with $u_j=5$ and $\hat x^i_j=3$, the state $(s,t)=(2,4)$ leaves opponent load $t-s=2$, so the deviation fits and contributes $3c_j/5$ to the deviation cost, whereas the state $(1,4)$ leaves opponent load $3$, so the same deviation exceeds the capacity and activates the big-$M$ term. The total load $t$ alone cannot distinguish these two cases; it is the joint state $(s,t)$ that supplies the opponent load $t-s$, and this is why the encoding is indexed by both. Details are in \S\ref{app:ISBP CEI}.

\subsubsection{VEST in ISBP}
\label{sss:vest-isbp}

In an ISBP game, the cost of player $i$ on bin $j$ is
$f^i_j(x^i_j) = c_j\, x^i_j/(x^{-i}_j + x^i_j)$,
with the convention $f^i_j(0) = 0$, so that $f^i_j$ is well defined on $[0,\,u_j - x^{-i}_j]$ also when $x^{-i}_j = 0$.
When bin $j$ is shared ($x^{-i}_j > 0$), $f^i_j$ is strictly concave. When bin $j$ is solo ($x^{-i}_j = 0$), $f^i_j$ equals $c_j$ for all $x^i_j > 0$ and $0$ for $x^i_j = 0$: it is constant, hence affine, on the positive part of the interval, but not on the whole interval when both zero and positive usage are feasible.
Therefore, Assumption~\ref{ass:strictly-concave} does not hold in general for ISBP games, and Corollary~\ref{thm:strong-vest} does not apply directly.
We instead apply Corollary~\ref{thm:weak-vest-separable}. From the definitions of $S^i_+$ and
$S^i_0$, for every free bin $j\in\free(x^i)$,
\[
  j\in S^i_+(x^i)\iff x^{-i}_j>0\ \text{(shared free bin)},\qquad
  j\in S^i_0(x^i)\iff x^{-i}_j=0\ \text{(solo free bin)}.
\]
By Corollary~\ref{thm:weak-vest-separable}, every direction $d$
within the minimal face of $P^i(x^{-i})$ containing $x^i$ satisfies $d_j = 0$ for all shared free bins $j \in S^i_+(x^i)$.
Since $P^i(x^{-i})$ is a bounded simplex ($A^i = \mathbf{1}^\top$), such directions also satisfy $\sum_{j \in \free(x^i)} d_j = 0$, hence $\sum_{j \in S^i_0(x^i)} d_j = 0$. This yields the following corollary, which generalizes Lemma~2 of \citet{hao2025integer}, which assumes uniform capacities and identical bin costs, to arbitrary instances.

\begin{corollary}
\label{ISBP: c1}
Let $x$ be a PNE of an ISBP instance, allowing non-uniform capacities. If player $i \in [n]$
uses two unsaturated bins $j$ and $k$ in $x$, then neither $j$ nor $k$
is shared with any other player.
\end{corollary}

The resulting \emph{weak VEST inequalities} forbid a player from using a shared unsaturated bin whenever she uses another unsaturated bin; they are valid for every PNE, and their one-hot formulation is given in \S\ref{app:ISBP CEI}, inequalities~\eqref{eq:supp-isbp-wvest1}--\eqref{eq:supp-isbp-wvest2}.

Although Corollary~\ref{thm:strong-vest} does not apply to ISBP, in uniform instances every PNE can be transformed into a vertex PNE with the same individual costs.

\begin{theorem}
\label{ISBP: t1}
In any uniform ISBP instance, for every PNE $x$, there exists a PNE
$\tilde{x}$ such that $\tilde{x}^i$ is a vertex of
$P^i(\tilde{x}^{-i})$ for every $i\in[n]$, and
\begin{equation}
    \pi^i(\tilde{x})=\pi^i(x) \qquad \forall i\in[n].
\end{equation}
\end{theorem}
\begin{proof}
Let $u$ be the common capacity. For an ISBP profile $y$, the free coordinates of player $i$ are the bins $j$ with $y^i_j>0$ and $y_j<u$, and since $P^i(y^{-i})$ is a simplex, $y^i$ is a vertex of $P^i(y^{-i})$ if and only if $i$ has at most one free coordinate (Lemma~\ref{lem:simplex-vertex}). Among the PNEs with the same cost vector as $x$, choose $\tilde x$ minimizing the number $p(\tilde x)$ of players whose strategy is not a vertex; we show that $p(\tilde x)=0$. Suppose not, and let $i$ be a nonvertex player, so that its set $S$ of free coordinates has at least two elements. By Corollary~\ref{ISBP: c1}, every bin in $S$ is solo.

Let $r:=\sum_{j\in S}(u-\tilde x^i_j)>0$ be the residual capacity of $S$, and let $j^*\in S$ have maximum cost. If $r\geq u$, then $\sum_{j\in S\setminus\{j^*\}}(u-\tilde x^i_j)=r-(u-\tilde x^i_{j^*})\geq\tilde x^i_{j^*}$, so player $i$ could move all units from $j^*$ to the other bins of $S$, which remain solo, and reduce her cost by $c_{j^*}>0$, contradicting equilibrium. Hence $0<r<u$. Define $x'$ by $x'^{\,i}_j=u$ for $j\in S\setminus\{j^*\}$, $x'^{\,i}_{j^*}=u-r$, and $x'=\tilde x$ otherwise. Since $\sum_{j\in S}\tilde x^i_j=|S|u-r$, $x'^{\,i}$ is feasible; its only free coordinate is $j^*$, so it is a vertex; and all bins in $S$ remain positive and solo, so $\pi^i(x')=\pi^i(\tilde x)$ and, as $\tilde x^{-i}$ is unchanged, $x'^{\,i}$ is still a best response.

Now fix $h\neq i$. Player $h$ uses no bin in $S$, so $\pi^h(x')=\pi^h(\tilde x)$. Suppose $\hat x^h$ is a profitable deviation from $x'$. If it uses no bin in $S$, it is profitable from $\tilde x$ as well. Otherwise it places $q\leq r$ units in $S$, all on $j^*$ since the other bins in $S$ are saturated at $x'$, at cost $c_{j^*}q/(u-r+q)$. Distributing these $q$ units among the bins of $S$ at $\tilde x$, whose residual capacity is $r$, leaves every used bin with load at least $u-r+q$; since $c_j\leq c_{j^*}$ on $S$, the resulting deviation costs at most $c_{j^*}q/(u-r+q)$ on $S$ and the same as $\hat x^h$ elsewhere, so it is profitable from $\tilde x$. Both cases contradict that $\tilde x$ is a PNE, so $x'$ is a PNE with the same cost vector as $\tilde x$.

Finally, every $h\neq i$ has the same free coordinates at $x'$ and $\tilde x$, so $p(x')=p(\tilde x)-1$, contradicting the choice of $\tilde x$.
\end{proof}

\paragraph{Strong VEST cut.}
By Theorem~\ref{ISBP: t1}, in uniform instances some optimal PNE has every player using at most one unsaturated bin, so we can add the following cut to the GZR master problem:
\begin{equation}
\label{eq:vest-cut}
\sum_{j \in [m]} \sum_{t=1}^{u_j-1} \sum_{s=1}^{\min\{t,w_i\}}
y^{j,t}_{i,s} \;\le\; 1 \qquad \forall\, i \in [n],
\end{equation}
which enforces that each player uses at most one unsaturated bin.
Unlike the weak VEST inequalities, inequality~\eqref{eq:vest-cut} is optimality-preserving rather than valid for every PNE: it may cut off PNE profiles, but by Theorem~\ref{ISBP: t1} some PNE with the same individual (hence social) costs satisfies it. It therefore preserves objectives that depend only on player costs and should not be imposed under arbitrary allocation objectives.

\subsection{Network Formation Games (NFG)}
\label{subsec:nfg}

\subsubsection{Game Definition}
Network formation is a standard cost-sharing benchmark: \citet{dragotto2023zero} evaluate Zero-Regret on a game in the style of \citet{chen2009network}, where each player routes one unit along a single source--sink path and each edge cost is split equally among its users, as in the uncapacitated network design game of \citet{anshelevich2008price}. That model is uncapacitated, hence an IPG rather than a GIPG. We therefore add edge capacities, which couple the feasible sets, and let each player split an integer demand across paths, with costs shared in proportion to flow. The classical potential-function and price-of-stability results do not survive these changes: an NFG on parallel edges between a common source and sink is an ISBP game, so the model contains instances without a PNE.

We restrict attention to directed acyclic networks, as in the layered networks studied computationally. Every feasible nonnegative flow then decomposes into source--sink paths, so $x^i_e\leq D_i$. This justifies the individual-flow ranges in the one-hot formulation. A model allowing directed circulations would require larger state ranges.

A capacitated NFG consists of a directed acyclic network $G = (V, E)$ with edge
capacities $u_e \in \Z_+$ and edge costs $c_e > 0$, and players
$[n]$ each with origin-destination pair $(o_i, d_i)$ and integer demand
$D_i \in \Z_+$. Player $i$ routes $D_i$ units of integer-splittable
flow from $o_i$ to $d_i$; let $x^i_e \in \Z_{\ge 0}$ denote the flow
of player $i$ on edge $e$, and $x_e = \sum_k x^k_e$ the total flow.
Under proportional flow-based cost sharing, player $i$'s cost is:
\[
  \pi^i(x) \;=\; \sum_{e \in E} c_e \cdot \frac{x^i_e}{x_e},
\]
with the convention $c_e x^i_e / x_e = 0$ whenever $x^i_e = 0$.

Given opponents' flows, player $i$ minimizes $\sum_{e\in E}c_e x^i_e/(x^{-i}_e+x^i_e)$ over integer $o_i$--$d_i$ flows of value $D_i$ satisfying $x^{-i}_e+x^i_e\le u_e$ for every edge.

\subsubsection{One-hot Encoding and Lifted Formulation}
\label{ss:nfg-form}

The cost structure is identical to ISBP, with edges replacing bins and demand replacing weight. We reuse the one-hot consistency and lifted-cost construction described in Section~\ref{ss:isbp-form}, together with node-balance equations. Against fixed $x^{-i}$, separation is a separable concave-cost integer-flow problem with per-edge cost $c_e s/(x^{-i}_e+s)$ and residual capacity $u_e-x^{-i}_e$. The detailed formulation is in \S\ref{app: NFG formulation}.

\subsubsection{CEI Encoding and Activation}
\label{ss:nfg-cei}

The NFG CEI is the edge-indexed analogue of the ISBP construction: its activation term is triggered when the deviation exceeds the residual capacity of an edge it uses. \S\ref{app: NFG CEI} gives the explicit formula~\eqref{eq:supp-nfg-cei}.

\subsubsection{VEST in NFG}
\label{sss:vest-nfg}

At the edge level the cost has the same curvature structure as in ISBP: shared free edges are strictly concave coordinates and solo free edges are locally affine. Directions within the minimal face of the conditional flow polytope are circulations supported on the free-edge subgraph (directed acyclicity of $G$ does not preclude undirected cycles, traversed with signs). Corollary~\ref{thm:weak-vest-separable} therefore implies that at a PNE every cycle of the undirected free-edge subgraph of player $i$ consists entirely of solo edges; equivalently, every shared free edge is a bridge of that subgraph. The cycle inequalities~\eqref{eq:supp-nfg-vest} in \S\ref{app: NFG CEI} express this condition and are valid for every cycle. Our implementation adds them a priori for a fixed fundamental cycle basis, which yields a valid subfamily but does not enforce the full condition; enforcing it completely would require separating violated cycles in the incumbent's free-edge subgraph.

\subsection{Generalized Knapsack Games (GKG and GKGB)}
\label{subsec:gkg}

\subsubsection{Game Definitions}

Generalized knapsack games are the class on which \citet{duguet2025branch1} evaluate their method, and we use them in two forms: the linear game GKG below, and GKGB, which adds the reciprocally bilinear interactions of \citet{dragotto2023zero} and is not a potential game.

\begin{definition}[Generalized Knapsack Game]
\label{def:gkg}
A GKG consists of $n$ players and $m$ items.
Player $i$ has a private $0$-$1$ knapsack with profits $p_{ij} \geq 0$, weights $w_{ij}$, and capacity $b_i$: $\calX^i_{\mathrm{private}} = \{x^i \in \{0,1\}^m : \sum_j w_{ij} x^i_j \leq b_i\}$.
Items have shared capacities: $\sum_i x^i_j \leq c_j$ for all $j$, where $c_j \in \{1, \ldots, n\}$.
Player $i$ minimizes $\pi^i(x) = P_i - \sum_j p_{ij} x^i_j$, where $P_i := \sum_j p_{ij}$ ensures $\pi^i \geq 0$.
\end{definition}

Since each player's cost depends only on her own variables, a jointly feasible unilateral deviation changes the social cost $F(x)=\sum_i\pi^i(x)$ by exactly the deviator's cost change; hence $F$ is an exact potential, every social optimum is a PNE, and the POS equals one whenever the minimum social cost is positive. GKG is therefore an exact potential game, and its socially optimal PNE can also be obtained directly by a binary MILP that maximizes total profit subject to the private knapsack and shared capacity constraints. Performance on this family alone does not isolate the value of CEI separation for equilibrium selection in non-potential games, a point we return to in Section~\ref{subsec:comp:duguet}; GKGB and the cost-sharing games serve that purpose. GKG nevertheless allows a comparison with an existing algorithm on a class where the optimal equilibrium is known, and it shows how quickly CEI separation reaches it, how many cuts it needs, and how this scales in the number of players, which we examine up to $n=50$.

GKGB extends GKG with \emph{reciprocally bilinear interactions} from \citet{dragotto2023zero}.
In addition to the private knapsack profits, each player's cost depends on opponents' item selections through pairwise interactions.

\begin{definition}[Generalized Knapsack Game with Bilinear Costs]
\label{def:gkgb}
A GKGB is a GKG (Definition~\ref{def:gkg}) in which each player~$i$'s cost includes bilinear interaction terms:
\[
\pi^i(x) = P_i - \sum_j p_{ij} x^i_j - \sum_{k \neq i} \sum_j C^i_{kj}\, x^i_j\, x^k_j,
\]
where $C^i_{kj} \in \R$ captures the externality that player~$k$'s selection of item~$j$ imposes on player~$i$.
Positive coefficients represent synergistic effects; negative coefficients represent adversarial externalities.
In GKGB the offset is enlarged relative to GKG,
\[
P_i \;:=\; \sum_j p_{ij} \;+\; \sum_{k \neq i} \sum_j \max\{C^i_{kj},\, 0\},
\]
which ensures $\pi^i \geq 0$ also in the presence of positive interaction coefficients.
\end{definition}

The offset $P_i$ is a constant for each player, so it changes neither exact best responses nor exact equilibria. It does change multiplicative approximation factors and cost-based welfare ratios, and it is therefore part of the game definition for these quantities. Nonnegativity of the costs is required by Definition~\ref{def:alpha-pne} and by the monotonicity of $\alpha$-PNE in $\alpha$ (Section~\ref{sec:tas}).

Neither GKG nor GKGB requires a one-hot encoding of resource loads, in contrast with ISBP and NFG, where the proportional cost-sharing term makes one necessary (Sections~\ref{ss:isbp-form} and~\ref{ss:nfg-form}). Binary player--item variables represent strategies, binary fullness indicators encode deviation feasibility, and in GKGB McCormick variables linearize the bilinear terms.

\subsubsection{CEI Encoding}

Because GKG and GKGB use binary player--item variables, their CEIs require no one-hot lift.
For a fixed deviation, binary fullness indicators identify whether the selected items have sufficient residual capacity.
The GKG deviation cost is affine; in GKGB, fixing the deviation makes its opponent-interaction term affine, while standard McCormick variables linearize the incumbent bilinear terms.
\S\ref{app:gkg-cei} gives the indicator system, CEIs, and big-$M$ values.

\section{Computational Study}
\label{sec:experiments}

\subsection{Experimental Design}
\label{subsec:exp-design}

We evaluate GZR on the four game classes, with up to two experiments per class. The experiments have three aims: to assess whether GZR computes socially optimal equilibria and certifies nonexistence within a practical time budget; to assess how tight the certified bounds produced by $\alpha$-bisection are where no exact PNE exists; and to measure the effect of the structural cuts.

\textbf{Instances.}
We test on 842 regular and 842 non-regular ISBP instances (all the enumerated tuples with $n \in \{4,\ldots,10\}$ players, $m \in \{n-2,\ldots,n+3\}$ bins, and item weight $w \in \{4,\ldots,15\}$, with the capacity $u$ derived from divisibility conditions subject to $mu - nw \geq 0$; heterogeneous weights, capacities, and costs for the non-regular family), 128 NFG instances on layered networks under tight and loose capacity regimes, the $1{,}260$ GKG benchmark instances published by \citet{duguet2025branch1} plus a 156-instance scalability study of our own with up to $n = 50$ players, and 72 GKGB instances (Types B and C of \citealt{dragotto2023zero}, plus an interpolating Type BC), with shared capacities added throughout.
Full generation details appear in \S\ref{app:instance-generation}.

\textbf{Experiment~1: Best PNE and Price of Stability.}
We first run RRR-BRD (Section~\ref{sec:tas}) with up to $R = 20$ rounds of sequential best-response updates per restart, visiting players in uniformly random order.
If some round produces a profile with no improving deviation for any player, that profile is certified as an exact PNE and the heuristic stops; otherwise, among the end-of-round profiles, the one with the smallest approximation factor $\alpha(x)$ is selected.
This candidate warm-starts GZR, which minimizes social cost over the $\alpha$-PNE set; dividing by the social optimum gives the POS.

\textbf{Experiment~2: $\alpha$-Bisection Search (ABS).}
For instances where Experiment~1 finds no exact PNE---GZR returns \textsc{Infeasible} or times out without confirming one---ABS (Section~\ref{sec:tas}) searches for the smallest certified approximation factor, using the RRR-BRD value $\alpha(x)$ as the initial upper bound and a per-call GZR limit of 10 minutes (600\,s). It is invoked only for ISBP and GKGB, because Experiment~1 finds a PNE on every GKG and NFG instance.
The bisection stops at width $10^{-3}$ for ISBP and $10^{-2}$ for GKGB, the looser value reflecting the higher cost of each GZR call on the McCormick-linearized GKGB formulation. Independently of the stopping rule, a certified interval $[\alpha^{\mathrm{LB}},\alpha^{\mathrm{UB}}]$ is reported as \emph{verified} when its width is below $\varepsilon=10^{-2}$; this reporting threshold gives the Verified column the same meaning for both games, and both endpoints are certified regardless of it. The strong VEST cut~\eqref{eq:vest-cut} and the player-anchoring constraint (\S\ref{app:symmetry}) are valid only for exact equilibria and are therefore imposed only in the initial test at $\alpha=1$, where Remark~\ref{rmk:existence-preserving} shows that an \textsc{Infeasible} outcome remains a valid nonexistence certificate; every call at $\alpha>1$ omits them, whereas the bin-load ordering constraint is valid at every $\alpha$ and is kept throughout. Deviations generated across iterations are carried over (Section~\ref{sec:tas}).

\textbf{Implementation.}
Our implementation, instance generators, and scripts reproducing all tables are publicly available at \url{https://github.com/HyunwooLee0429/GIPG} and will be deposited with a DOI upon acceptance, in accordance with the journal's software and data policy.

All experiments are implemented in Python using Gurobi~13.0 \citep{gurobi} with lazy constraint callbacks for CEI separation, 16 threads for both the master and the best-response subproblems, and a GZR time limit of 10 minutes (600\,s) per instance; generation and restart seeds are fixed and recorded per instance in the repository, each instance run once, on a single machine with an Intel Core Ultra~7 265F processor (2.40\,GHz) and 32\,GB RAM.

\textbf{Solver settings.}
The master problem is solved with \texttt{LazyConstraints}~$=1$ and the social-optimum reference model, which needs no separation, with \texttt{LazyConstraints}~$=0$; all tolerances are Gurobi defaults (\texttt{MIPGap}~$=10^{-4}$, \texttt{IntFeasTol}~$=10^{-5}$, \texttt{FeasibilityTol}~$=10^{-6}$). A status of \textsc{Optimal} therefore certifies optimality within the default gap rather than in exact arithmetic; the largest terminal gap among \textsc{Optimal} instances is $1.7\times10^{-5}$ for ISBP and $10^{-4}$ for GKG and GKGB. Best-response subproblems are solved to optimality in milliseconds, and a violation is declared when the incumbent's regret exceeds $10^{-9}$. A CEI generated by a feasible deviation is valid whether or not that deviation is an exact best response, so inexact best responses can only weaken the acceptance test, never produce invalid cuts. The nonexistence certificates rest on the solver's infeasibility proofs at these tolerances and were not re-verified in exact arithmetic; for $(n,m,w,u)=(5,4,5,8)$, nonexistence is also known analytically \citep{hao2025integer}.

\textbf{Performance metrics.}
Tables report instance counts; PNEs found by BRD and GZR; the average BRD approximation ratio $\alpha_{\mathrm{brd}}$; average GZR wall-clock time, POS (best-PNE cost divided by social-optimum cost), and number of separated CEIs; and the numbers ending as \textsc{Optimal}, \textsc{Infeasible}, or \textsc{Time}\allowbreak\_\allowbreak\textsc{Limit}; the last is the \textsc{Unresolved} status of Algorithm~\ref{alg:gzr}, which in our experiments always results from the time limit.
POS is averaged only over instances on which GZR returned a PNE, so rows with time-limit terminations condition on the easier instances; $\alpha_{\mathrm{brd}}$ is averaged over all instances, using the best end-of-round profile of each.
Interpreting a ratio as the POS requires both a certified equilibrium and optimality certificates for both terms. The social-optimum reference model solved to optimality on all but 23 of the $3{,}300$ instances (one non-regular ISBP, two NFG, and 20 GKG scalability instances); on those, the incumbent reference value is used, and the resulting ratio is neither a guaranteed upper nor lower bound on the true POS.

Table~\ref{tab:result-summary} collects the principal aggregate results; the explicit application-specific CEIs and VEST inequalities are in \S\ref{app:explicit-formulations}, and the full computational files are available in the repository.
\begin{table}[t]
\caption{Aggregate results (times in seconds, averaged). OPT/INF/TL count instances ending \textsc{Optimal}, \textsc{Infeasible}, and \textsc{Time\_Limit}. The GKG benchmark row uses the Experiment~1 protocol (warm start, 16 threads, optimality certified); the single-threaded first-equilibrium comparison with \citet{duguet2025branch1} is reported separately in Section~\ref{subsec:comp:duguet} and Table~\ref{tab:gkg-benchmark-cells}. POS averages for ISBP non-regular, NFG, and GKG (scalability) include 1, 2, and 20 incumbent-reference instances (Section~\ref{subsec:exp-design}).}
\label{tab:result-summary}
\centering
\small
\begin{tabular}{lrrrrrrrr}
\toprule
Game/group & Count & BRD PNE & GZR PNE & $\alpha_{\mathrm{brd}}$ & $T_{\rm GZR}$ & POS & $\#_{\mathrm{CEI}}$ & OPT/INF/TL\\
\midrule
ISBP regular & 842 & 819 & 820 & 1.001 & 8.4 & 1.000 & 89.6 & 818/15/9\\
ISBP non-regular & 842 & 841 & 842 & 1.000 & 6.1 & 1.000 & 19.5 & 839/0/3\\
NFG loose & 64 & 64 & 64 & 1.000 & 37.4 & 1.005 & 0.7 & 63/0/1\\
NFG tight & 64 & 64 & 64 & 1.000 & 94.9 & 1.015 & 3.0 & 58/0/6\\
GKG (benchmark) & 1260 & 1260 & 1260 & 1.000 & 0.15 & 1.000 & 5.6 & 1260/0/0\\
GKG (scalability) & 156 & 156 & 156 & 1.000 & 126.2 & 1.000 & 90.9 & 129/0/27\\
GKGB Type B & 24 & 24 & 24 & 1.000 & 68.0 & 1.012 & 294.5 & 22/0/2\\
GKGB Type BC & 24 & 15 & 16 & 1.002 & 386.4 & 1.057 & 4111.1 & 9/0/15\\
GKGB Type C & 24 & 8 & 11 & 1.007 & 386.6 & 1.072 & 4015.7 & 8/1/15\\
\bottomrule
\end{tabular}
\end{table}

\subsection{Results}
\label{subsec:comp:results}
Table~\ref{tab:result-summary} summarizes Experiment~1.
For ISBP, GZR finds a PNE in 820 of 842 regular and in all 842 non-regular instances, with average solve times below nine seconds, and certifies nonexistence on 15 regular instances; the 22 regular instances on which no exact PNE was found (15 certified nonexistent and 7 unresolved at the time limit) proceed to ABS.
Median times are far below the means (regular/non-regular ISBP 0.2/0.3\,s, NFG 1.1/2.9\,s, GKG scalability 8.5\,s, benchmark 0.04\,s), reflecting heavy right tails from the time-limit instances.
On regular instances the master problem is strengthened by the symmetry-breaking constraints of \S\ref{app:symmetry}: player anchoring requires $w \leq u$ and is imposed on exactly the 627 of 842 instances satisfying it; bin-load ordering and the VEST cut~\eqref{eq:vest-cut} are applied throughout.

RRR-BRD and GZR find a PNE in all 128 NFG instances; average GZR time rises from 37.4 seconds under loose capacities to 94.9 seconds under tight capacities, while the average POS remains small, at 1.005 and 1.015 respectively.
In the 156-instance GKG scalability study through 50 players, GZR finds a PNE in every instance with POS $1.000$ throughout; difficulty grows sharply with weight--profit correlation (average times of 3.3, 72.3, and 303.0 seconds across the three correlation levels), and the 27 time-limit terminations lack only an optimality certificate.
GKGB is the hardest class, and its difficulty increases with the share of adversarial (negative) interaction coefficients: GZR finds a PNE in all 24 Type~B, 16 of 24 Type~BC, and 11 of 24 Type~C instances, with average POS among the equilibria found of 1.012, 1.057, and 1.072. This is consistent with \citet{dragotto2023zero}, who observe that mixed-sign interactions both hinder computation and can preclude existence.

\subsection{Approximate Equilibria}
\label{subsec:comp:abs}
All 22 ISBP instances entering ABS are regular. We recall that neither a time-limited GZR run nor cycling of best-response dynamics establishes nonexistence; the 15 \textsc{Infeasible} outcomes at $\alpha=1$ do, and for $(5,4,5,8)$ nonexistence was also established analytically by \citet{hao2025integer}.
Table~\ref{tab:isbp-tas} in \S\ref{app:abs-certificates} reports the instance-level results (Experiment~1 status, initial bound $\alpha_{\mathrm{init}}$, certified bounds, verification at $\varepsilon = 10^{-2}$, iterations, and end-to-end time, which includes the warm start, the $\alpha = 1$ test, and the final lower-bound call and may therefore exceed $\mathrm{Iters} \times 600$\,s).

Of the 22 instances, 7 are fully verified. Of the remaining 15, seven lack a certified lower bound entirely, because the lower-bound GZR calls time out, particularly at larger item sizes ($w \geq 11$), and eight close to certified widths between $0.013$ and $0.037$, just outside the threshold. Lower bounds are harder to certify than upper bounds because the structural cuts are dropped once $\alpha$ exceeds one (Section~\ref{subsec:exp-design}), so infeasibility proofs at $\alpha>1$ are obtained on a weaker master.
The tightest certified upper bounds range from $\alpha^{\mathrm{UB}} = 1.009$ to $1.081$, with most below $1.05$.
ABS performs 5--7 bisection iterations per instance, but on the hard instances most are time-limit terminations that advance the working interval without producing any certificate (Section~\ref{sec:tas}): on $(7,6,8,11)$, seven iterations and $3{,}629$ seconds---roughly $518$ seconds per call against the 600-second limit---move the certified upper bound only once, from $1.073$ to $1.037$, and certify no new lower bound.
Iteration counts should therefore not be read as convergence; the certified interval widths are the meaningful summary.
For GKGB, 21 instances (8 Type~BC, 13 Type~C) enter ABS; 6 are fully verified, the rest lack a certified lower bound, and the certified upper bounds range from $1.000$ to $1.034$. Five of the six are verified because the larger warm start of ABS found an exact PNE that Experiment~1 had missed, which raises the number of Type~BC and Type~C instances with a known PNE to 20 and 12 of 24.
Approximate equilibria within a few percent of exact, with certificates of their distance from the best achievable factor, are thus attainable even without an exact PNE.

\subsection{Comparison with Prior Intersection-Cut Methods}
\label{subsec:comp:duguet}
We run GZR on the $1{,}260$ generalized knapsack instances published by \citet{duguet2025branch1}, so that both methods are evaluated on identical data. Their branch-and-cut looks for one equilibrium or a proof that none exists, so we match that task: GZR is given no RRR-BRD warm start and stops at the first pure equilibrium, and, as their experiments use a single core, we run it on one thread.

Under a ten-minute limit GZR finds a PNE on all $1{,}260$ instances, averaging $0.066$ seconds and $22.9$ separated CEIs. Under a one-hour limit their method solves $471$, averaging $2{,}337$ seconds and $41{,}275$ cuts (per-cell counts, hardware, solvers, and limits are in Table~\ref{tab:gkg-benchmark-cells}). On the $471$ instances both resolve, the median ratio of their time to ours is $139$; over all instances the median ratio of cut counts is $1{,}325$.

The cut counts differ even more. \citet{duguet2025branch1} attribute their growth in the generalized setting to the local validity of the intersection cuts; the maximum on these instances is $233{,}309$. CEIs are globally valid by Proposition~\ref{prop:cei}, so each separated inequality is kept for the rest of the search, and GZR needed at most $172$ on these instances. The two counts are not directly comparable, since intersection cuts are separated at fractional node solutions and CEIs at integer candidates; the difference is that a locally valid cut is discarded when its subtree closes, whereas a CEI never is.

Optimizing over the equilibrium set is a task their method does not address: with the warm start and the master solved to gap zero, GZR returns the socially optimal PNE on all $1{,}260$ instances in $0.15$ seconds on average, at a price of stability of $1.000$. One caveat remains: as noted in Section~\ref{subsec:gkg}, a GKG social optimum is automatically a PNE, so this family measures the overhead of separating redundant equilibrium constraints (the reference MILP averages $0.06$ seconds against $0.15$ for GZR), while equilibrium selection is tested on ISBP, NFG, and GKGB.

\subsection{Impact of Structural Cuts}
\label{subsec:comp:vest}
Table~\ref{tab:vest-comparison} (\S\ref{app:vest-impact}) compares CEI--GZR configurations with and without the VEST-based restrictions on matched instances: on regular ISBP, the strong VEST cut~\eqref{eq:vest-cut} reduces average time from 13.0 to 8.4 seconds---a $1.55\times$ aggregate speedup---and time-limit terminations from 16 to 9, with the largest gains on the hardest instances.
The weak VEST inequalities, in contrast, do not pay off on these families: on non-regular ISBP the aggregate speedup is 0.888 with an essentially neutral per-instance ratio of 1.001, and on NFG both metrics favor the configuration without cuts (0.923 and 0.865). The NFG cycle inequalities are a large static family added a priori (\S\ref{app: NFG CEI}) and are carried whether or not they prune, which explains the overhead.

The aggregate speedup compares total runtimes, and the per-instance speedup is the mean of the off/on runtime ratios over instances solved to optimality in both configurations. Since the two configurations use independent RRR-BRD warm starts, these figures do not fully separate the effect of the restrictions from variation in initialization; they support the family-specific conclusion that the optimality-preserving strong VEST cut helps on regular ISBP, whereas the weaker valid inequalities do not.

\section{Conclusion}
\label{sec:conclusion}
We have introduced conditional equilibrium inequalities, an affine representation of the equilibrium conditions of finite integer games with shared constraints whose activation terms distinguish feasible from infeasible deviations. Since incumbent and deviation costs are expressed in the same master variables, the inequalities are globally valid and can be separated by best-response computations. The resulting GZR algorithm optimizes a linear objective over exact or approximate equilibria, and the accompanying bisection procedure produces certified bounds on the smallest achievable approximation factor.

The structural results play a complementary role. Concavity of the costs and integrality of the conditional polytopes imply that each player's cost is constant on the minimal face containing her equilibrium strategy, yielding vertex restrictions under strict concavity and coordinate restrictions under separability; in uniform ISBP a cost-preserving transformation justifies the vertex restriction even though solo-bin costs are not strictly concave. They also clarify which inequalities preserve every equilibrium and which preserve only an optimal one, a distinction that matters whenever nonexistence or optimality is certified.

The computational study shows that GZR computes socially optimal equilibria and certifies nonexistence across cost-sharing and knapsack games, and that added structure does not automatically improve running times: the strong VEST cut helps on regular ISBP, whereas the weaker valid inequalities do not. The 16 certified nonexistence results, 15 of them for regular bin packing instances, are to our knowledge the first such certificates obtained computationally for these games. For the cost-sharing games, the size of the one-hot lift limits the capacities and demands that can be handled directly; reducing it, separating rather than enumerating the cycle inequalities, and equilibrium-aware branching rules are natural directions for future work.

\bibliography{GIPG}

\begin{thebibliography}{27}
\providecommand{\natexlab}[1]{#1}
\providecommand{\url}[1]{\texttt{#1}}
\expandafter\ifx\csname urlstyle\endcsname\relax
  \providecommand{\doi}[1]{doi: #1}\else
  \providecommand{\doi}{doi: \begingroup \urlstyle{rm}\Url}\fi

\bibitem[Anshelevich et~al.(2008)Anshelevich, Dasgupta, Kleinberg, Tardos, Wexler, and Roughgarden]{anshelevich2008price}
Elliot Anshelevich, Anirban Dasgupta, Jon Kleinberg, {\'E}va Tardos, Tom Wexler, and Tim Roughgarden.
\newblock The price of stability for network design with fair cost allocation.
\newblock \emph{SIAM Journal on Computing}, 38\penalty0 (4):\penalty0 1602--1623, 2008.

\bibitem[Bil{\`o}(2006)]{bilo2006packing}
Vittorio Bil{\`o}.
\newblock On the packing of selfish items.
\newblock In \emph{Proceedings of the 20th IEEE International Parallel \& Distributed Processing Symposium}, 2006.

\bibitem[Carvalho(2016)]{carvalho2016computation}
M.~Carvalho.
\newblock \emph{Computation of equilibria on integer programming games}.
\newblock PhD thesis, Universidade do Porto, Faculdade de Ci{\^e}ncias, 2016.

\bibitem[Carvalho et~al.(2022)Carvalho, Lodi, and Pedroso]{carvalho2022computing}
M.~Carvalho, A.~Lodi, and J.~Pedroso.
\newblock Computing equilibria for integer programming games.
\newblock \emph{European Journal of Operational Research}, page S0377221722002727, March 2022.
\newblock ISSN 03772217.

\bibitem[Carvalho et~al.(2018)Carvalho, Lodi, and Pedroso]{carvalho2018complexity}
Margarida Carvalho, Andrea Lodi, and Jo{\~a}o~Pedro Pedroso.
\newblock Existence of {N}ash equilibria on integer programming games.
\newblock In \emph{Operational Research}, pages 11--23. Springer, 2018.

\bibitem[Carvalho et~al.(2023)Carvalho, Dragotto, Lodi, and Sankaranarayanan]{carvalho2023integer}
Margarida Carvalho, Gabriele Dragotto, Andrea Lodi, and Sriram Sankaranarayanan.
\newblock Integer {P}rogramming {G}ames: {A} {G}entle {C}omputational {O}verview.
\newblock In \emph{Tutorials in Operations Research: Advancing the Frontiers of OR/MS: From Methodologies to Applications}, pages 31--51. INFORMS, 2023.

\bibitem[Carvalho et~al.(2026)Carvalho, Dragotto, Lodi, and Sankaranarayanan]{carvalho2026cutandplay}
Margarida Carvalho, Gabriele Dragotto, Andrea Lodi, and Sriram Sankaranarayanan.
\newblock The {Cut-and-Play} algorithm: Computing {Nash} equilibria via outer approximations.
\newblock \emph{Operations Research}, 74\penalty0 (2):\penalty0 1070--1086, 2026.
\newblock Published online July 8, 2025.

\bibitem[Chassein et~al.(2016)Chassein, Krumke, and Thielen]{Chassein2016}
Andr\'{e} Chassein, Sven~O. Krumke, and Clemens Thielen.
\newblock Capacitated network design games with weighted players.
\newblock \emph{Networks.}, 68\penalty0 (2):\penalty0 141–158, September 2016.
\newblock ISSN 0028-3045.

\bibitem[Chen and Roughgarden(2009)]{chen2009network}
Ho-Lin Chen and Tim Roughgarden.
\newblock Network design with weighted players.
\newblock In \emph{Theory of Computing Systems}, volume~45, pages 302--324, 2009.

\bibitem[Cr{\"o}nert and Minner(2024)]{cronert2024equilibrium}
Tobias Cr{\"o}nert and Stefan Minner.
\newblock Equilibrium identification and selection in finite games.
\newblock \emph{Operations Research}, 72\penalty0 (2):\penalty0 816--831, 2024.
\newblock \doi{10.1287/opre.2022.2413}.

\bibitem[Del~Pia et~al.(2017)Del~Pia, Ferris, and Michini]{delpia2017tu}
Alberto Del~Pia, Michael~C. Ferris, and Carla Michini.
\newblock Totally unimodular congestion games.
\newblock \emph{SIAM Journal on Discrete Mathematics}, 31\penalty0 (2):\penalty0 875--881, 2017.

\bibitem[Dragotto and Scatamacchia(2023)]{dragotto2023zero}
Gabriele Dragotto and Rosario Scatamacchia.
\newblock The zero regrets algorithm: Optimizing over pure nash equilibria via integer programming.
\newblock \emph{INFORMS Journal on Computing}, 35\penalty0 (5):\penalty0 1143--1160, 2023.

\bibitem[Duguet et~al.(2025{\natexlab{a}})Duguet, Harks, Schmidt, and Schwarz]{duguet2025branch2}
AloANGs Duguet, Tobias Harks, Martin Schmidt, and Julian Schwarz.
\newblock Branch-and-cut for computing approximate equilibria of mixed-integer generalized nash games.
\newblock \emph{arXiv preprint arXiv:2511.03340}, 2025{\natexlab{a}}.

\bibitem[Duguet et~al.(2025{\natexlab{b}})Duguet, Harks, Schmidt, and Schwarz]{duguet2025branch1}
AloANGSs Duguet, Tobias Harks, Martin Schmidt, and Julian Schwarz.
\newblock Branch-and-cut for mixed-integer generalized nash equilibrium problems.
\newblock \emph{arXiv preprint arXiv:2506.02520}, 2025{\natexlab{b}}.

\bibitem[Fabiani and Grammatico(2020)]{fabiani2020vehicle}
Filippo Fabiani and Sergio Grammatico.
\newblock Multi-vehicle automated driving as a generalized mixed-integer potential game.
\newblock \emph{IEEE Transactions on Intelligent Transportation Systems}, 21\penalty0 (3):\penalty0 1064--1073, 2020.
\newblock \doi{10.1109/TITS.2019.2901505}.

\bibitem[Feldman and Ron(2012)]{feldman2012capacitated}
Michal Feldman and Tom Ron.
\newblock Capacitated network design games.
\newblock In Maria Serna, editor, \emph{Algorithmic Game Theory: 5th International Symposium, SAGT 2012, Barcelona, Spain, October 22--23, 2012, Proceedings}, volume 7615 of \emph{Lecture Notes in Computer Science}, pages 132--143, Berlin, Heidelberg, 2012. Springer.
\newblock ISBN 978-3-642-33996-7.

\bibitem[Fotakis et~al.(2005)Fotakis, Kontogiannis, and Spirakis]{fotakis2005selfish}
Dimitris Fotakis, Spyros Kontogiannis, and Paul Spirakis.
\newblock Selfish unsplittable flows.
\newblock \emph{Theoretical Computer Science}, 348\penalty0 (2-3):\penalty0 226--239, 2005.

\bibitem[Hao and Michini(2025)]{hao2025integer}
Bainian Hao and Carla Michini.
\newblock Integer-splittable bin packing games.
\newblock \emph{Optimization Online}, 2025.
\newblock URL \url{https://optimization-online.org/2025/05/integer-splittable-bin-packing-games/}.
\newblock Full version; extended abstract in Proceedings of the 21st Conference on Web and Internet Economics (WINE), 2025. The results cited in this paper (Lemmas~1, 2, and~4) are proved in the full version.

\bibitem[Harks and Schwarz(2025)]{harks2025generalized}
Tobias Harks and Julian Schwarz.
\newblock Generalized {N}ash equilibrium problems with mixed-integer variables.
\newblock \emph{Mathematical Programming}, 209:\penalty0 231--277, 2025.

\bibitem[K{\"o}ppe et~al.(2011)K{\"o}ppe, Ryan, and Queyranne]{koppe2011rational}
Matthias K{\"o}ppe, Christopher~Thomas Ryan, and Maurice Queyranne.
\newblock Rational generating functions and integer programming games.
\newblock \emph{Operations research}, 59\penalty0 (6):\penalty0 1445--1460, 2011.

\bibitem[Lee et~al.(2024)Lee, Hildebrand, Cai, and B{\"u}y{\"u}ktahtak{\i}n]{lee2024random}
Hyunwoo Lee, Robert Hildebrand, Wenbo Cai, and {\.I}~Esra B{\"u}y{\"u}ktahtak{\i}n.
\newblock Random-restart best-response dynamics for large-scale integer programming games and their applications.
\newblock \emph{arXiv preprint arXiv:2409.04078}, 2024.

\bibitem[Rosen(1965)]{rosen1965existence}
J~Ben Rosen.
\newblock Existence and uniqueness of equilibrium points for concave n-person games.
\newblock \emph{Econometrica: Journal of the Econometric Society}, pages 520--534, 1965.

\bibitem[Rosenthal(1973)]{rosenthal1973class}
Robert~W. Rosenthal.
\newblock A class of games possessing pure-strategy nash equilibria.
\newblock \emph{International Journal of Game Theory}, 2\penalty0 (1):\penalty0 65--67, 1973.

\bibitem[Sagratella(2016)]{sagratella2016computing}
Simone Sagratella.
\newblock Computing all solutions of nash equilibrium problems with discrete strategy sets.
\newblock \emph{SIAM Journal on Optimization}, 26\penalty0 (4):\penalty0 2190--2218, 2016.

\bibitem[Sagratella(2017)]{sagratella2017potential}
Simone Sagratella.
\newblock Algorithms for generalized potential games with mixed-integer variables.
\newblock \emph{Computational Optimization and Applications}, 68\penalty0 (3):\penalty0 689--717, 2017.
\newblock \doi{10.1007/s10589-017-9927-4}.

\bibitem[Sagratella(2019)]{sagratella2019coupling}
Simone Sagratella.
\newblock On generalized {Nash} equilibrium problems with linear coupling constraints and mixed-integer variables.
\newblock \emph{Optimization}, 68\penalty0 (1):\penalty0 197--226, 2019.
\newblock \doi{10.1080/02331934.2018.1545125}.

\bibitem[Schwarze and Stein(2023)]{schwarze2023branch}
Silvia Schwarze and Oliver Stein.
\newblock A branch-and-prune algorithm for discrete {N}ash equilibrium problems.
\newblock \emph{Computational Optimization and Applications}, 86\penalty0 (2):\penalty0 491--519, 2023.

\end{thebibliography}

\appendix
\numberwithin{equation}{section}
\numberwithin{table}{section}
\numberwithin{figure}{section}

\section{Omitted Proofs}
\label{app:moved-proofs}

This section of the e-companion contains the proofs omitted from the main text, together with a remark on equilibrium-preserving restrictions used in Section~\ref{sec:experiments}.

\subsection{Proof of Proposition~\ref{prop:cei}}\label{app:proof-cei}
\begin{proof}
At a lifted $\alpha$-PNE, a feasible deviation has zero activation and its CEI follows from the equilibrium condition and the affine cost identities on $\mathcal H_{\mathrm{MI}}$. For an infeasible deviation, the right-hand side is at least $M_i$, which bounds the incumbent cost. Thus every CEI holds. Conversely, a feasible profitable deviation from a nonequilibrium profile has zero activation, and its exact affine cost representation makes its CEI violated. This proves exactness.

Under the additional representation hypothesis, the lifted equilibrium set is finite and rational, so its convex hull is a rational polytope whose vertices belong to the set. Convexity of $\mathcal H$ and linearity of the CEIs imply that all hull points satisfy both; a hull point also in $\Z^K\times W$ is therefore a lifted equilibrium by exactness. A linear objective on a nonempty polytope attains an optimum at a vertex.
\end{proof}

\subsection{Proof of Proposition~\ref{prop:gzr-correct}}\label{app:proof-gzr}
\begin{proof}
For an accepted profile $x$, the best-response checks give
$\pi^i(x)\leq\alpha\min_{x^i\in\calX_i(x^{-i})}\pi^i(x^i,x^{-i})$ for every player, which is precisely the equilibrium condition. Every initial or separated cut preserves all equilibria. A complete master optimization therefore establishes objective optimality of its best accepted profile; a complete infeasibility search proves that no equilibrium exists. Neither conclusion follows from an unresolved search alone.
\end{proof}

\begin{remark}[Equilibrium-preserving restrictions]\label{rmk:existence-preserving}
If initial restrictions jointly preserve at least one equilibrium whenever one exists, a completed infeasibility search still proves nonexistence. To preserve an optimal objective value, they must jointly preserve an optimal equilibrium. Theorem~\ref{ISBP: t1} supplies cost-preserving representatives in uniform ISBP; Proposition~\ref{prop:simultaneous} justifies the simultaneous use of vertex, bin-ordering, and player-anchoring restrictions in regular instances with $w\leq u$. Player anchoring permutes individual costs, so this combination preserves social cost but need not preserve a player-specific objective. Numerical certification is a separate issue from these mathematical preservation statements.
\end{remark}

\subsection{Vertices of a bounded simplex}
\label{app:proof-simplex-vertex}
The following lemma is used in the proofs of Corollary~\ref{ISBP: c1} and Theorem~\ref{ISBP: t1}.
\begin{lemma}[Vertices of a bounded simplex]
\label{lem:simplex-vertex}
Suppose $A^i = \mathbf{1}^\top$, so that $P^i(x^{-i})$ is a bounded simplex
(as in ISBP). Then $x^i \in P^i(x^{-i})$ is a vertex of $P^i(x^{-i})$ if and
only if $|\free(x^i)| \leq 1$.
\end{lemma}
\begin{proof}
If distinct coordinates $j$ and $\ell$ are free at $x^i$, then, for all
sufficiently small $\varepsilon>0$,
\[
  x^i\pm\varepsilon(e_j-e_\ell)\in P^i(x^{-i}).
\]
These distinct points have midpoint $x^i$, so $x^i$ is not a vertex.
Conversely, suppose $|\free(x^i)|\leq1$ and
$x^i=\lambda y+(1-\lambda)z$ with $y,z\in P^i(x^{-i})$ and
$\lambda\in(0,1)$. Each coordinate of $x^i$ at a bound forces the same
coordinate of $y$ and $z$ to equal that bound. Thus they can differ only
on the one possible free coordinate, where
$\mathbf1^\top y=\mathbf1^\top z$ forces equality. Hence
$y=z=x^i$, proving that $x^i$ is a vertex.
\end{proof}

\subsection{Proof of Corollary~\ref{thm:weak-vest-separable}}
\label{app:proof-separable-vest}

\begin{proof}
Fix $i$ and a feasible direction $d$ within the minimal face $F$
containing $x^i$, and set $S=\free(x^i)$. Since
$x^i\in\operatorname{ri}(F)$, for some $\varepsilon>0$,
$x^i+td\in F$ whenever $|t|<\varepsilon$. Coordinates outside $S$ are
fixed on $F$, so Theorem~\ref{thm:weak-vest} implies that
\[
  \sum_{j\in S}h_j(t),\qquad
  h_j(t):=f^i_j(x^i_j+td_j;x^{-i}),
\]
is constant on this interval. For any $a,b$ in the interval and
$\lambda\in(0,1)$, the concavity gaps
\[
  h_j(\lambda a+(1-\lambda)b)
  -\lambda h_j(a)-(1-\lambda)h_j(b)
\]
are nonnegative and sum to zero; hence each vanishes. Thus every $h_j$
is affine. If $j\in S^i_+(x^i)$ and $d_j\neq0$, reparametrizing by
$s=x^i_j+td_j$ would make $f^i_j(\cdot;x^{-i})$ affine on a neighborhood
of $x^i_j$, contrary to the definition of $S^i_+(x^i)$. Therefore
$d_j=0$ for every $j\in S^i_+(x^i)$.
\end{proof}

\subsection{Proof of Corollary~\ref{ISBP: c1}}
\label{app:proof-isbp-c1}

\begin{proof}
For $r\in\{j,k\}$,
\[
  0<x^i_r<u_r-x^{-i}_r,
\]
so $j$ and $k$ are free and Lemma~\ref{lem:simplex-vertex} makes $x^i$
a nonvertex. If one of them, say $\ell$, were shared, then
$\ell\in S^i_+(x^i)$. With $q$ denoting the other bin,
$d=e_\ell-e_q$ is a feasible direction within the minimal face
containing $x^i$, since both coordinates are free. But
Corollary~\ref{thm:weak-vest-separable} requires $d_\ell=0$, contradicting
$d_\ell=1$. Hence neither bin is shared.
\end{proof}

\section{Instance Generation}
\label{app:instance-generation}

This section records how the instances used in Section~\ref{sec:experiments} are generated. All generators are seeded, and the scripts that produce them accompany the code in the repository.

\subsection{Integer splittable bin packing}
We generate two families of ISBP instances in the model of \citet{hao2025integer}.
\emph{Regular instances} have identical item weights $w$, bin capacities $u$, and
bin costs $c_j=w$; since the game is invariant under a common scaling of the bin
costs, this normalization affects neither the equilibria nor the reported ratios.
We enumerate all tuples $(n,m,w,u)$ with $n\in\{4,\ldots,10\}$ players,
$m\in\{n-2,\ldots,n+3\}$ bins, $w\in\{4,\ldots,15\}$, and $u$ taken from the four
smallest values consistent with $mu-nw\geq0$ and $m\leq\lceil nw/u\rceil$, so that
$m=\lceil nw/u\rceil$ in every instance. This yields 842 tuples.
These conditions do not exclude $w>u$: of the 842 regular instances, 215 have $w>u$, so the
player-anchoring constraint of \S\ref{app:symmetry}, which requires $w\leq u$, is
imposed only on the remaining 627.
\emph{Non-regular instances} introduce 20\% heterogeneity ($\delta=0.2$) around the
mean: item weights $w_i$, bin capacities $u_j$, and bin costs $c_j$ are drawn
uniformly and independently from
$[\lceil\bar w(1-\delta)\rceil,\lfloor\bar w(1+\delta)\rfloor]$ for the weights and
from $[\lceil\bar u(1-\delta)\rceil,\lfloor\bar u(1+\delta)\rfloor]$ for both the
capacities and the costs. When the drawn capacities do not accommodate the total
demand, the smallest ones are incremented until $\sum_ju_j=\sum_iw_i$; this occurs
on 220 instances, which therefore have no spare capacity.
RRR-BRD initializes each restart from a random feasible profile; on tight
instances with $\sum_ju_j=\sum_iw_i$ the random sampler may fail, in which case
the initial profile is built by a greedy sequential fill, which is always feasible
because items are splittable. All 842 tuples of each family are reported.

\subsection{Network formation games}
The instances of \citet{dragotto2023zero} are uncapacitated and hence outside our class, so we generate our own (\S\ref{subsec:nfg}).
We generate 128 NFG instances on layered directed networks.
Each network has $L \in \{3, 5, 8, 10\}$ layers with $W \in \{3, 5, 8, 10\}$ nodes per layer.
Inter-layer edges are drawn independently with probability $p = 0.8$, with at least one edge per node to ensure source-to-sink connectivity, and costs $c_e \sim U\{1, 20\}$.
Each of $n \in \{2, 4, 6, 8\}$ players routes integer-splittable flow from a common source in the first layer to a common sink in the last layer, with individual demands $d_i \sim U\{2, 5\}$.
Edge capacities are set relative to the total demand $D = \sum_i d_i$ in two regimes: \emph{tight}, with $\kappa_e \sim U[\max(d_{\max}, \lceil 0.7 D \rceil), \lceil 0.9 D \rceil]$, and \emph{loose}, with $\kappa_e \sim U[\max(d_{\max}, \lceil 0.9 D \rceil), \lceil 1.1 D \rceil]$; capacities are raised until all players can route simultaneously. The design is 2 capacity levels $\times$ 4 player counts $\times$ 4 layer counts $\times$ 4 widths $= 128$ instances.

\subsection{Generalized knapsack games}
We follow the Pisinger-like generation scheme of \citet{duguet2025branch1}.
For each player $i$ and item $j$, weights are drawn independently as $w_{ij} \sim U\{1, R\}$ with $R = 1{,}000$.
Profits depend on the \emph{correlation type}:
(i)~\emph{uncorrelated}: $p_{ij} \sim U\{1, R\}$, independent of weights;
(ii)~\emph{weakly correlated}: $p_{ij} \sim U\{\max(1,\, w_{ij} - R/10),\; w_{ij} + R/10\}$;
(iii)~\emph{strongly correlated}: $p_{ij} = w_{ij} + R/10$ (deterministic).
Knapsack capacities are $b_i = \max(1, \lfloor \mathit{cf} \sum_{j} w_{ij} \rfloor)$ with \emph{capacity factor} $\mathit{cf} \in \{0.2, 0.5\}$, and shared item capacities are $c_j \sim U\{1, n\}$.

We evaluate GKG in two settings. The first is the benchmark of \citet{duguet2025branch1}---$n \in \{2, 3, 4\}$ players, $m \in \{5, 10, 15, 20, 30, 40, 50\}$ items, two capacity factors, three correlation types and 10 replications, for $3 \times 7 \times 2 \times 3 \times 10 = 1{,}260$ instances---for which we use their published instance files rather than regenerating them, so that Section~\ref{subsec:comp:duguet} compares the two methods on identical data. Their generator differs from the scheme above in two respects: it draws from $R = 100$, and it raises each knapsack capacity to $\max(\lfloor \mathit{cf} \sum_j w_{ij} \rfloor, \max_j w_{ij})$, so that every item fits on its own.

The second is a scalability study with $n \in \{2, 3, 5, 8, 10, 15, 20, 25, 30, 35, 40, 45, 50\}$ players, $m \in \{30, 50\}$ items, the same three correlation types, and 2 replications, yielding $13 \times 2 \times 3 \times 2 = 156$ instances; the capacity factor is not crossed here but assigned one value per replication, so that each $(n, m, \mathrm{corr})$ cell contains one instance at $\mathit{cf} = 0.2$ and one at $\mathit{cf} = 0.5$.

\subsection{Generalized knapsack games with bilinear coefficients}
GKGB adds to the knapsack profits the \emph{reciprocally bilinear} interaction terms $\sum_{k \neq i} \sum_{j} C^i_{kj} x^i_j x^k_j$ of \citet{dragotto2023zero}.
We adopt their Type~B and Type~C datasets, which define interactions between $n \in \{2, 3\}$ players over $m \in \{25, 50, 75, 100\}$ items with budget ratios in $\{2, 5, 8\}$: Type~B has positive coefficients $C^i_{kj} \sim U[1, 100]$ and Type~C mixed ones $C^i_{kj} \sim U[-100, 100]$.
We add \emph{Type~BC}, with $C^i_{kj} \sim U[-50, 100]$, interpolating between the two.
To convert each IPG into a GNEP we impose shared per-item capacities $\sum_i x^i_j \leq c_j$ with $c_j \sim U\{1, n\}$ (seed~42), yielding 72 instances, 24 per type. The master formulation is given in \S\ref{app:explicit-formulations}.

\section{Explicit Formulations}
\label{app:explicit-formulations}

This section of the e-companion records the explicit application-specific formulas referenced in
Section~\ref{sec:games}.

\subsection{ISBP formulation}\label{app: ISBP formulation}
For each player $i$, bin $j$, and integers $0\leq t\leq u_j$, $0\leq s\leq\min\{t,w_i\}$, let $y^{j,t}_{i,s}$ indicate the joint state $(x^i_j,x_j)=(s,t)$. Let $z^j_t$ indicate total load $t$. The continuous relaxation of the following constraints, together with the recovery equations below, defines $\mathcal H$; the binary requirements specify $W$.

\begin{subequations}
\label{eq:isbp-H}
\begin{align}
& \sum_{t=0}^{u_j}\sum_{s=0}^{\min\{t,w_i\}} y^{j,t}_{i,s}=1
&& \forall i\in [n],\; j\in [m],
\label{eq:isbp-c1} \\
& \sum_{j\in [m]}\sum_{t=0}^{u_j}\sum_{s=0}^{\min\{t,w_i\}} s\,y^{j,t}_{i,s}=w_i
&& \forall i\in [n],
\label{eq:isbp-c2} \\
& \sum_{s=0}^{\min\{t,w_i\}} y^{j,t}_{i,s}\le z^j_t
&& \forall i\in [n],\; j\in [m],\; t=0,\ldots,u_j,
\label{eq:isbp-c3} \\
& \sum_{t=0}^{u_j} z^j_t=1
&& \forall j\in [m],
\label{eq:isbp-c4} \\
& \sum_{i\in [n]}\sum_{s=0}^{\min\{t,w_i\}} s\,y^{j,t}_{i,s}
= t\, z^j_t
&& \forall j\in [m],\; t=0,\ldots,u_j,
\label{eq:isbp-c5} \\
& y^{j,t}_{i,s}\in\{0,1\}
&& \forall i\in [n],\; j\in [m],\; t=0,\ldots,u_j,\; s=0,\ldots,\min\{t,w_i\},
\label{eq:isbp-c6} \\
& z^j_t\in\{0,1\}
&& \forall j\in [m],\; t=0,\ldots,u_j.
\label{eq:isbp-c7}
\end{align}
\end{subequations}
Constraints~\eqref{eq:isbp-c1}--\eqref{eq:isbp-c5} select one state per player--bin pair, enforce demand, and make the selected total load equal the sum of individual contributions.

To retain the original strategy variables in $\mathcal H$, include the recovery equations
\[
x^i_j=\sum_{t=0}^{u_j}\sum_{s=0}^{\min\{t,w_i\}}s\,y^{j,t}_{i,s},\qquad x_j=\sum_i x^i_j.
\]
For each player $i$, the lifted cost $\tilde{\pi}^i$ is:
\begin{equation}
\label{eq:isbp-cost}
\tilde{\pi}^i \;=\; \sum_{j\in [m]}
\sum_{t=1}^{u_j}\sum_{s=1}^{\min\{t,w_i\}}
c_j \frac{s}{t}\, y^{j,t}_{i,s},
\end{equation}
which is affine in the $y$ variables and equals the original nonlinear
cost $\pi^i$ at every point of $\mathcal H_{\mathrm{MI}}$.

The encoding uses $O(nmuw)$ binary variables for $u=\max_j u_j$ and $w=\max_iw_i$. It replaces nonlinear costs by an exact affine representation, at a size that grows with capacities and demands.

\subsection{ISBP CEI formulation}
\label{app:ISBP CEI}
Fix player $i$ and a deviation $\hat{x}^i$ with
$\sum_j\hat{x}^i_j=w_i$, $0\leq\hat{x}^i_j\leq u_j$, and support
$\hat B_i=\{j:\hat{x}^i_j>0\}$.  For $j\in\hat B_i$, define
\begin{align*}
\mathcal F_j(\hat{x}^i_j)
 &=\{(s,t):0\leq t\leq u_j,\ 
       \max\{0,t-u_j+\hat{x}^i_j\}\leq s\leq\min\{t,w_i\}\},\\
\mathcal I_j(\hat{x}^i_j)
 &=\{(s,t):u_j-\hat{x}^i_j+1\leq t\leq u_j,\ 
       0\leq s\leq\min\{t-u_j+\hat{x}^i_j-1,w_i\}\}.
\end{align*}
The state-indexed CEI is
\begin{equation}
\sum_{j\in[m]}\sum_{t=1}^{u_j}\sum_{s=1}^{\min\{t,w_i\}}
 c_j\frac{s}{t}y^{j,t}_{i,s}
\leq \alpha\sum_{j\in\hat B_i}
 \sum_{(s,t)\in\mathcal F_j(\hat{x}^i_j)}
 c_j\frac{\hat{x}^i_j}{(t-s)+\hat{x}^i_j}y^{j,t}_{i,s}
 +M\sum_{j\in\hat B_i}
 \sum_{(s,t)\in\mathcal I_j(\hat{x}^i_j)}y^{j,t}_{i,s},
\label{eq:supp-isbp-cei}
\end{equation}
with $M=\sum_jc_j$.  The last term deactivates the inequality precisely when
the deviation exceeds residual capacity on a used bin.

For the weak VEST inequalities, let
\[
g_{i,j}=\sum_{t=1}^{u_j-1}\sum_{s=1}^{\min\{t-1,w_i\}}y^{j,t}_{i,s},\qquad
h_{i,j}=\sum_{t=1}^{\min\{u_j-1,w_i\}}y^{j,t}_{i,t},\qquad
f_{i,j}=g_{i,j}+h_{i,j},\qquad G_i=\sum_jg_{i,j}.
\]
The inequalities are
\begin{align}
G_i&\leq1 &&\forall i\in[n],\label{eq:supp-isbp-wvest1}\\
g_{i,j_1}+f_{i,j_2}&\leq1
 &&\forall i\in[n],\ j_1\neq j_2.\label{eq:supp-isbp-wvest2}
\end{align}

\subsection{NFG formulation}\label{app: NFG formulation}
Use the ISBP encoding with edges replacing bins, $w_i$ replaced by $D_i$, and $u_j$ by $u_e$. The individual-flow cap $D_i$ is valid on the directed acyclic networks specified in Section~\ref{subsec:nfg}. Retain~\eqref{eq:isbp-c1} and~\eqref{eq:isbp-c3}--\eqref{eq:isbp-c7}, and recover $x^i_e=\sum_{t=0}^{u_e}\sum_{s=0}^{\min\{t,D_i\}}s\,y^{e,t}_{i,s}$. Replace the ISBP demand equation~\eqref{eq:isbp-c2} by node balance:

\begin{equation}
\label{eq:nfg-flow}
\sum_{e \in \delta^+(v)} \sum_{t,s} s\, y^{e,t}_{i,s}
\;-\;
\sum_{e \in \delta^-(v)} \sum_{t,s} s\, y^{e,t}_{i,s}
\;=\;
\begin{cases}
  D_i & v = o_i, \\
  -D_i & v = d_i, \\
  0 & \text{otherwise},
\end{cases}
\qquad \forall i \in [n],\ v \in V,
\end{equation}
where each sum over $(t,s)$ uses $0\leq t\leq u_e$ and $0\leq s\leq\min\{t,D_i\}$, and $\delta^+(v)$ and $\delta^-(v)$ denote the sets of edges leaving
and entering node $v$. Together with the constraints inherited from
ISBP, these define the NFG mixed-integer formulation; relaxing the binary domains gives $\mathcal H$. The lifted
cost $\tilde{\pi}^i$ is:
\[
  \tilde{\pi}^i \;=\; \sum_{e \in E}
  \sum_{t=1}^{u_e}\sum_{s=1}^{\min\{t,D_i\}}
  c_e \frac{s}{t}\, y^{e,t}_{i,s},
\]
which is affine in the $y$ variables and equals $\pi^i$ at every point of $\mathcal H_{\mathrm{MI}}$.

\subsection{NFG conditional equilibrium and cycle inequalities}
\label{app: NFG CEI}
For a feasible $o_i$--$d_i$ deviation $\hat{x}^i$ of value $D_i$, let
$\hat E_i=\{e:\hat{x}^i_e>0\}$ and define
\begin{align*}
\mathcal F_e(\hat{x}^i_e)
 &=\{(s,t):0\leq t\leq u_e,\ 
       \max\{0,t-u_e+\hat{x}^i_e\}\leq s\leq\min\{t,D_i\}\},\\
\mathcal I_e(\hat{x}^i_e)
 &=\{(s,t):u_e-\hat{x}^i_e+1\leq t\leq u_e,\ 
       0\leq s\leq\min\{t-u_e+\hat{x}^i_e-1,D_i\}\}.
\end{align*}
The NFG CEI is the edge-indexed inequality
\begin{equation}
\sum_{e\in E}\sum_{t=1}^{u_e}\sum_{s=1}^{\min\{t,D_i\}}
 c_e\frac{s}{t}y^{e,t}_{i,s}
\leq \alpha\sum_{e\in\hat E_i}
 \sum_{(s,t)\in\mathcal F_e(\hat{x}^i_e)}
 c_e\frac{\hat{x}^i_e}{(t-s)+\hat{x}^i_e}y^{e,t}_{i,s}
 +M\sum_{e\in\hat E_i}
 \sum_{(s,t)\in\mathcal I_e(\hat{x}^i_e)}y^{e,t}_{i,s}.
\label{eq:supp-nfg-cei}
\end{equation}
Here $M=\sum_{e\in E}c_e$ bounds any player cost. For an undirected cycle $C$, define
\[
g_{i,e}=\sum_{t=1}^{u_e-1}\sum_{s=1}^{\min\{t-1,D_i\}}y^{e,t}_{i,s},
\quad G_i^C=\sum_{e\in C}g_{i,e},
\quad H_i^C=\sum_{e\in C}\sum_{t=1}^{\min\{u_e-1,D_i\}}y^{e,t}_{i,t}.
\]
For any such cycle, the valid inequality is
\begin{equation}
g_{i,e}+\bigl(G_i^C+H_i^C-|C|\bigr)\leq0
\qquad \forall i\in[n],\ C,\ e\in C.
\label{eq:supp-nfg-vest}
\end{equation}
When all edges of $C$ are free for player $i$, the inequality forces $g_{i,e}=0$. If some edge is not free, the inequality is automatic. The implementation adds inequalities for a fixed fundamental cycle basis a priori; the CEI count refers to separated equilibrium inequalities. This basis supplies only a subfamily: a violating cycle in the incumbent's free-edge subgraph need not be a basis cycle, and basis cycles may contain non-free edges. Enforcing the full condition requires separation on that free-edge subgraph.

\subsection{GKG and GKGB CEI formulations}
\label{app:gkg-cei}
Binary variables $x_{ij} \in \{0,1\}$ represent player--item decisions.
Feasibility indicators $z^{\mathrm{full}}_{ij} \in \{0,1\}$ track whether item $j$ is at capacity from player $i$'s perspective:
\begin{align}
\sum_{k \neq i} x_{kj} &\leq (c_j - 1) + (n - 1) z^{\mathrm{full}}_{ij} \label{eq:gkg-zfull-a} \\
\sum_{k \neq i} x_{kj} &\geq c_j \, z^{\mathrm{full}}_{ij}. \label{eq:gkg-zfull-b}
\end{align}

\textbf{GKG.} Take $M_i=P_i$ in the following inequality (with $M=M_i$).
For player $i$ and deviation $\xhat^i$, the CEI is:
\begin{equation}
\label{eq:gkg-cei}
P_i - \sum_j p_{ij} x_{ij} \leq \alpha \left( P_i - \sum_j p_{ij} \xhat^i_j \right) + M \sum_{j : \xhat^i_j = 1} z^{\mathrm{full}}_{ij}.
\end{equation}
When the deviation is feasible ($z^{\mathrm{full}}_{ij} = 0$ for all $j$ with $\xhat^i_j = 1$), this enforces the $\alpha$-equilibrium condition.
When infeasible, the big-$M$ term deactivates the constraint.

\textbf{GKGB.}
For a fixed deviation $\xhat^i$, its interaction terms are affine in the opponents' master variables. The incumbent interaction terms use the McCormick variables specified below.
The feasibility indicators $z^{\mathrm{full}}_{ij}$ carry over unchanged.
For player~$i$ and deviation $\xhat^i$, the CEI is:
\begin{equation}
\label{eq:gkgb-cei}
P_i - \sum_j p_{ij} x_{ij} - \sum_{k \neq i} \sum_j C^i_{kj}\, x_{ij}\, x_{kj}
\leq \alpha \left( P_i - \sum_j p_{ij} \xhat^i_j - \sum_{k \neq i} \sum_j C^i_{kj}\, \xhat^i_j\, x_{kj} \right) + M \sum_{j : \xhat^i_j = 1} z^{\mathrm{full}}_{ij}.
\end{equation}
In the master MILP, each bilinear term $x^i_j \cdot x^k_j$ is linearized exactly on binary strategies by the McCormick inequalities with auxiliary variable $y^{ik}_j$ and constraints $y^{ik}_j \leq x^i_j$, $y^{ik}_j \leq x^k_j$, $y^{ik}_j \geq x^i_j + x^k_j - 1$, $y^{ik}_j \geq 0$.
For GKGB we take $M_i = P_i + \sum_{k \neq i} \sum_j \max\{-C^i_{kj},\, 0\}$, an upper bound on any realized cost $\pi^i$.

\section{Symmetry and Compatibility in ISBP}
\label{app:symmetry}
\paragraph{Bin-load ordering.}
Partition bins into classes of identical capacity and cost. For a class $D=\{j_1,\ldots,j_{|D|}\}$, impose
\begin{equation}\label{eq:bsb}
\sum_{t=0}^{u_{j_\ell}}t z^{j_\ell}_t\geq
\sum_{t=0}^{u_{j_{\ell+1}}}t z^{j_{\ell+1}}_t
\qquad(\ell=1,\ldots,|D|-1).
\end{equation}
\begin{theorem}[Bin-load ordering]\label{thm:bsb}
For every $\alpha$-PNE of an ISBP instance, relabeling bins within identical classes gives an $\alpha$-PNE satisfying~\eqref{eq:bsb} and preserving every player's cost.
\end{theorem}
\begin{proof}
Sort the bins in each class by load. Relabel every player's allocation using the same permutation. Capacities, demands, and costs are preserved. Applying the inverse permutation to any feasible deviation from the relabeled profile gives a feasible deviation with the same cost from the original profile. Thus all $\alpha$-equilibrium comparisons are preserved.
\end{proof}
Permuting a candidate deviation within an identical-bin class likewise produces another admissible candidate and hence another valid CEI. This does not require the permuted deviation to be feasible against the current opponents; its activation term accounts for that possibility.

\paragraph{Player anchoring.}
For regular instances with common item size $w\leq u$, set $k=\lceil nw/u\rceil$. Then $k\leq n$, and joint feasibility gives $k\leq m$. Impose
\begin{equation}\label{eq:player-anchor}
\sum_{t=1}^{u}\sum_{s=1}^{\min\{t,w\}}y^{i,t}_{i,s}\geq1
\qquad(i=1,\ldots,k).
\end{equation}
\begin{theorem}[Player anchoring]\label{thm:player-anchor}
For every PNE of a regular instance with $w\leq u$, relabeling identical bins and players gives a PNE satisfying~\eqref{eq:bsb} and~\eqref{eq:player-anchor}. Individual costs are permuted and social cost is preserved.
\end{theorem}
\begin{proof}
Sort bins by load. By Lemma~1 of \citet{hao2025integer}, exactly $k$ bins are occupied. Form the bipartite graph joining these bins to the players using them. If a nonempty set $B$ of bins has at most $|B|-1$ neighboring players, then
\[
nw\leq(|B|-1)w+(k-|B|)u\leq(k-1)u<nw,
\]
a contradiction. Hall's theorem gives distinct player representatives for all occupied bins. Relabel representative $i$ as player $i$ and extend this injection to a player permutation. Player demands are identical, so equilibrium is preserved, costs are permuted, and bin loads remain ordered.
\end{proof}

\begin{proposition}[Simultaneous compatibility]\label{prop:simultaneous}
In a regular instance with $w\leq u$, every PNE has an equal-social-cost representative satisfying~\eqref{eq:vest-cut},~\eqref{eq:bsb}, and~\eqref{eq:player-anchor} simultaneously.
\end{proposition}
\begin{proof}
Apply Theorem~\ref{ISBP: t1} to obtain vertex strategies with the same individual costs, then apply Theorem~\ref{thm:player-anchor}. Bin and player permutations preserve the number of unsaturated bins used by each player, up to relabeling. Hence the vertex restriction remains satisfied, and the cost vector is only permuted.
\end{proof}
These restrictions preserve the minimum social cost among exact equilibria. Player anchoring need not preserve a player-specific weighted objective. No anchoring assertion for $w>u$, or for $\alpha>1$, is used here.

\section{ABS Certificates}
\label{app:abs-certificates}

Table~\ref{tab:isbp-tas} reports the instance-level ABS results for the 22 regular ISBP instances for which Experiment~1 found no exact PNE, discussed in Section~\ref{subsec:comp:abs}.
The corresponding GKGB results appear in Table~\ref{tab:gkgb-tas}: of the 21 instances, six are verified---five because the RRR-BRD warm start of ABS, which uses 5 restarts of 20 rounds instead of the 3 restarts of 15 rounds of Experiment~1, found an exact PNE, so $\alpha^{\mathrm{UB}}=1$ and the interval is closed by the trivial bound $\alpha_L=1$ without any bisection call---and one through bisection---while the rest never obtained a certified lower bound.

\begin{table}[h!]
\caption{ISBP Experiment~2: ABS results for instances with no PNE found in Experiment~1. The Verified column is evaluated at the common threshold $\varepsilon = 10^{-2}$ of Section~\ref{subsec:exp-design}. An entry $\alpha^{\mathrm{LB}} = 1.000$ means nonexistence of an exact PNE was certified at $\alpha = 1$; Iters counts bisection calls including those ending in a time limit, which advance the working interval without contributing a certificate (Section~\ref{sec:tas}).}

\label{tab:isbp-tas}
\centering
\small
\scalebox{0.85}{
\begin{tabular}{rrrrrrrrrrrr}
\toprule
Type & $n$ & $m$ & $w$ & $u$ & $\alpha{=}1$ & $\alpha_{\mathrm{init}}$ & $\alpha^{\mathrm{UB}}$ & $\alpha^{\mathrm{LB}}$ & Verified & Iters & $T_{\mathrm{ABS}}$ \\ 
\midrule
Reg & 5 & 4 & 5 & 8 & INF & 1.081 & 1.081 & 1.080 & TRUE & 7 & 7.8 \\
Reg & 5 & 4 & 7 & 11 & INF & 1.055 & 1.055 & 1.054 & TRUE & 6 & 30.6 \\
Reg & 5 & 4 & 8 & 13 & INF & 1.106 & 1.015 & 1.014 & TRUE & 7 & 31.4 \\
Reg & 5 & 4 & 9 & 14 & INF & 1.041 & 1.034 & 1.034 & TRUE & 6 & 45.7 \\
Reg & 5 & 4 & 10 & 16 & INF & 1.067 & 1.009 & 1.009 & TRUE & 7 & 91.0 \\
Reg & 5 & 4 & 11 & 17 & INF & 1.033 & 1.023 & 1.022 & TRUE & 6 & 246.9 \\
Reg & 5 & 4 & 13 & 20 & INF & 1.028 & 1.016 & 1.016 & TRUE & 5 & 605.3 \\
Reg & 7 & 5 & 11 & 19 & INF & 1.018 & 1.013 & 1.000 & FALSE & 5 & 2089.8 \\
Reg & 7 & 6 & 8 & 11 & INF & 1.073 & 1.037 & 1.000 & FALSE & 7 & 3629.1 \\
Reg & 7 & 6 & 11 & 15 & INF & 1.025 & 1.013 & 1.000 & FALSE & 5 & 2470.8 \\
Reg & 7 & 6 & 13 & 18 & TL & 1.047 & 1.020 & - & FALSE & 6 & 3877.5 \\
Reg & 7 & 6 & 14 & 19 & INF & 1.039 & 1.019 & 1.000 & FALSE & 6 & 2953.6 \\
Reg & 8 & 6 & 7 & 11 & INF & 1.055 & 1.032 & 1.000 & FALSE & 6 & 1372.3 \\
Reg & 8 & 6 & 9 & 14 & INF & 1.041 & 1.021 & 1.000 & FALSE & 6 & 3111.9 \\
Reg & 8 & 6 & 11 & 17 & INF & 1.033 & 1.029 & 1.000 & FALSE & 6 & 3277.5 \\
Reg & 8 & 6 & 12 & 19 & TL & 1.065 & 1.056 & - & FALSE & 7 & 3296.0 \\
Reg & 8 & 6 & 13 & 20 & TL & 1.028 & 1.021 & - & FALSE & 5 & 3539.9 \\
Reg & 8 & 6 & 14 & 22 & INF & 1.055 & 1.021 & 1.000 & FALSE & 6 & 3047.1 \\
Reg & 8 & 6 & 15 & 23 & TL & 1.024 & 1.024 & - & FALSE & 5 & 3620.5 \\
Reg & 9 & 8 & 11 & 14 & TL & 1.041 & 1.040 & - & FALSE & 6 & 3719.9 \\
Reg & 9 & 8 & 15 & 19 & TL & 1.044 & 1.044 & - & FALSE & 6 & 4233.0 \\
Reg & 10 & 8 & 12 & 17 & TL & 1.026 & 1.026 & - & FALSE & 5 & 3628.3 \\
\bottomrule
\end{tabular}
}
\end{table}

\begin{table}[h!]
\caption{GKGB Experiment~2: ABS results for the 21 instances on which Experiment~1 did not confirm a PNE. Verified applies the common threshold $\varepsilon = 10^{-2}$. Rows with $\alpha^{\mathrm{UB}}=1.000$ are instances on which the ABS warm start found an exact PNE; there $\alpha^{\mathrm{LB}}=1.000$ is the trivial bound $\alpha_L=1$ of Section~\ref{sec:tas}, not a nonexistence certificate.}
\label{tab:gkgb-tas}
\centering
\small
\scalebox{0.85}{
\begin{tabular}{rrrrrrrrrrr}
\toprule
Type & $n$ & $m$ & BG & $\alpha{=}1$ & $\alpha_{\mathrm{init}}$ & $\alpha^{\mathrm{UB}}$ & $\alpha^{\mathrm{LB}}$ & Verified & Iters & $T_{\mathrm{ABS}}$ \\
\midrule
Type BC & 2 & 75 & 2 & TL & 1.018 & 1.009 & - & FALSE & 1 & 1.0 \\
Type BC & 2 & 75 & 8 & TL & 1.000 & 1.000 & 1.000 & TRUE & 0 & 0.2 \\
Type BC & 2 & 100 & 5 & TL & 1.003 & 1.003 & - & FALSE & 0 & 1.0 \\
Type BC & 3 & 50 & 2 & TL & 1.022 & 1.011 & - & FALSE & 2 & 2504.6 \\
Type BC & 3 & 75 & 2 & TL & 1.000 & 1.000 & 1.000 & TRUE & 0 & 0.7 \\
Type BC & 3 & 100 & 2 & TL & 1.009 & 1.009 & - & FALSE & 0 & 1.2 \\
Type BC & 3 & 100 & 5 & TL & 1.000 & 1.000 & 1.000 & TRUE & 0 & 0.5 \\
Type BC & 3 & 100 & 8 & TL & 1.000 & 1.000 & 1.000 & TRUE & 0 & 0.7 \\
Type C & 2 & 50 & 2 & INF & 1.006 & 1.006 & 1.000 & TRUE & 0 & 0.9 \\
Type C & 2 & 75 & 2 & TL & 1.005 & 1.005 & - & FALSE & 0 & 0.9 \\
Type C & 2 & 75 & 5 & TL & 1.008 & 1.008 & - & FALSE & 0 & 0.9 \\
Type C & 2 & 100 & 5 & TL & 1.021 & 1.010 & - & FALSE & 2 & 1216.4 \\
Type C & 3 & 50 & 2 & TL & 1.058 & 1.022 & - & FALSE & 3 & 1285.1 \\
Type C & 3 & 50 & 5 & TL & 1.013 & 1.006 & - & FALSE & 1 & 466.7 \\
Type C & 3 & 50 & 8 & TL & 1.024 & 1.012 & - & FALSE & 2 & 1513.1 \\
Type C & 3 & 75 & 2 & TL & 1.024 & 1.012 & - & FALSE & 2 & 1532.2 \\
Type C & 3 & 75 & 5 & TL & 1.000 & 1.000 & 1.000 & TRUE & 0 & 0.3 \\
Type C & 3 & 75 & 8 & TL & 1.004 & 1.004 & - & FALSE & 0 & 2.5 \\
Type C & 3 & 100 & 2 & TL & 1.001 & 1.001 & - & FALSE & 0 & 1.5 \\
Type C & 3 & 100 & 5 & TL & 1.034 & 1.034 & - & FALSE & 2 & 2405.4 \\
Type C & 3 & 100 & 8 & TL & 1.006 & 1.006 & - & FALSE & 0 & 1.5 \\
\bottomrule
\end{tabular}
}
\end{table}

\Needspace{0.4\textheight}
\section{Impact of VEST Cuts: Detailed Comparison}
\label{app:vest-impact}

Table~\ref{tab:vest-comparison} reports the matched-instance comparison discussed in Section~\ref{subsec:comp:vest}.

\begin{table}[H]
\caption{Impact of VEST cuts: CEI--GZR with and without VEST, averaged over matched instances. Speedup ratios are computed from unrounded averages, so they may differ slightly from ratios of the rounded times shown.}
\label{tab:vest-comparison}
\centering
\small
\scalebox{0.85}{
\begin{tabular}{llccccccccc}
\toprule
Game & Cut & Count & $T_{\mathrm{GZR}}$ & $\#_{\mathrm{CEI}}$ & OPT & INF & TL & \multicolumn{2}{c}{Speedup} \\
\cmidrule(lr){9-10}
 & & & & & & & & Agg.\ & Per-inst.\ \\
\midrule
\multirow{2}{*}{ISBP Regular}
  & VEST on & 842 & 8.4 & 89.6 & 818 & 15 & 9 & \multirow{2}{*}{1.552} & \multirow{2}{*}{1.015} \\
  & VEST off & 842 & 13.0 & 84.9 & 817 & 9 & 16 & & \\
\midrule
\multirow{2}{*}{ISBP Non-regular}
  & wVEST on & 842 & 6.1 & 19.5 & 839 & 0 & 3 & \multirow{2}{*}{0.888} & \multirow{2}{*}{1.001} \\
  & wVEST off & 842 & 5.4 & 22.9 & 839 & 0 & 3 & & \\
\midrule
\multirow{2}{*}{NFG}
  & wVEST on & 128 & 66.1 & 1.8 & 121 & 0 & 7 & \multirow{2}{*}{0.923} & \multirow{2}{*}{0.865} \\
  & wVEST off & 128 & 61.1 & 1.9 & 121 & 0 & 7 & & \\
\bottomrule
\end{tabular}
}
\end{table}

\Needspace{0.4\textheight}
\section{GKG Benchmark Results by Cell}
\label{app:gkg-benchmark}

Table~\ref{tab:gkg-benchmark-cells} disaggregates by $(n,m)$ cell the benchmark of Section~\ref{subsec:comp:duguet}, in the cell layout used by \citet{duguet2025branch1}. Both methods are run on their published instance files. GZR finds a PNE in all 60 instances of every cell, with a maximum of $0.85$ seconds and $172$ separated CEIs over the whole set.

The two sides are not run on the same machine, and the settings are as follows. \citet{duguet2025branch1} use a single core of an Intel Xeon Gold 6126 at $2.6$\,GHz with $4$\,GB of RAM under a one-hour limit, calling Gurobi~12.0 from C++ with \texttt{FeasibilityTol}~$=10^{-9}$ and best responses solved at \texttt{MIPGap}~$=0$. We use one thread of an Intel Core Ultra~7 265F at $2.40$\,GHz with $32$\,GB under a ten-minute limit, calling Gurobi~13.0 from Python at the default tolerances of Section~\ref{subsec:exp-design}; this is the only single-threaded run in the paper, and the setting makes little difference here, as the mean time is $0.066$ seconds single-threaded against $0.072$ seconds with 16 threads. Their clock speed is the higher of the two and neither method approaches the memory available to it, so on the dimensions that can be matched the single-threaded comparison is a reasonably even one; the solver version and the language cannot be matched, and the one-hour limit favors their method.

\begin{table}[H]
\caption{GKG benchmark by $(n,m)$ cell, 60 instances per cell, on the instance files of \citet{duguet2025branch1}. For their method, \#S is the number of instances solved within one hour and $\overline{\mathrm{Cuts}}$ the mean number of intersection cuts. For CEI--GZR, all 60 are solved in every cell; $\bar{T}$ is the mean wall-clock seconds and $\overline{\mathrm{CEI}}$ the mean number of separated inequalities, single-threaded under a ten-minute limit.}
\label{tab:gkg-benchmark-cells}
\centering
\small
\scalebox{0.85}{
\begin{tabular}{r rr rr rr rr rr rr}
\toprule
 & \multicolumn{4}{c}{$n=2$} & \multicolumn{4}{c}{$n=3$} & \multicolumn{4}{c}{$n=4$} \\
\cmidrule(lr){2-5}\cmidrule(lr){6-9}\cmidrule(lr){10-13}
 & \multicolumn{2}{c}{Duguet et al.} & \multicolumn{2}{c}{CEI--GZR}
 & \multicolumn{2}{c}{Duguet et al.} & \multicolumn{2}{c}{CEI--GZR}
 & \multicolumn{2}{c}{Duguet et al.} & \multicolumn{2}{c}{CEI--GZR} \\
\cmidrule(lr){2-3}\cmidrule(lr){4-5}\cmidrule(lr){6-7}\cmidrule(lr){8-9}\cmidrule(lr){10-11}\cmidrule(lr){12-13}
$m$ & \#S & $\overline{\mathrm{Cuts}}$ & $\bar{T}$ & $\overline{\mathrm{CEI}}$
    & \#S & $\overline{\mathrm{Cuts}}$ & $\bar{T}$ & $\overline{\mathrm{CEI}}$
    & \#S & $\overline{\mathrm{Cuts}}$ & $\bar{T}$ & $\overline{\mathrm{CEI}}$ \\
\midrule
5 & 60 & 14 & 0.02 & 9.0 & 60 & 122 & 0.03 & 13.9 & 60 & 439 & 0.04 & 17.0 \\
10 & 60 & 768 & 0.03 & 11.7 & 54 & 21,662 & 0.04 & 18.2 & 40 & 71,818 & 0.06 & 24.4 \\
15 & 52 & 31,681 & 0.03 & 12.5 & 20 & 70,718 & 0.05 & 19.9 & 10 & 88,385 & 0.08 & 29.8 \\
20 & 16 & 59,234 & 0.03 & 14.0 & 13 & 72,471 & 0.06 & 22.8 & 3 & 74,203 & 0.09 & 30.5 \\
30 & 9 & 55,748 & 0.04 & 16.0 & 3 & 59,130 & 0.07 & 24.0 & 1 & 53,520 & 0.13 & 37.1 \\
40 & 3 & 45,585 & 0.06 & 19.3 & 1 & 43,821 & 0.08 & 26.8 & 2 & 32,333 & 0.16 & 41.8 \\
50 & 3 & 30,867 & 0.06 & 20.6 & 0 & 29,485 & 0.10 & 30.1 & 1 & 24,760 & 0.16 & 41.6 \\
\bottomrule
\end{tabular}
}
\end{table}

\end{document}